\documentclass{article}
\usepackage{graphicx} 

  \baselineskip 1 mm

\usepackage{graphicx}
\usepackage[top=1in, bottom=1in, left=1in, right=1in]{geometry}
\usepackage{amsmath}
\usepackage{amssymb}
\usepackage{tikz}
\usepackage[english]{babel}
\usepackage{amsthm,amsfonts,latexsym,amsopn,verbatim,amscd,amssymb}
\usepackage{hyperref}
\usepackage{pgfplots}
\usepackage{caption}
\usepackage{fancyhdr}
\usepackage{xfrac}

\newcommand{\bnum}{\begin{enumerate}}
\newcommand{\enum}{\end{enumerate}}
\newcommand{\equa}[3]

\theoremstyle{plain}
\newtheorem{theorem}{Theorem}[section]
\newtheorem{lemma}[theorem]{Lemma}
\newtheorem{proposition}[theorem]{Proposition}
\newtheorem{corollary}[theorem]{Corollary}

\theoremstyle{definition}

\newtheorem{Def}[theorem]{Definition}
\newtheorem{Example}[theorem]{Example}
\numberwithin{equation}{section}

 \usepackage{xcolor}

\usepackage{pifont}
\DeclareMathOperator*{\dprime}{\prime \prime}

\usepackage{hyperref}
\usepackage{pgfplots}
\usepackage{caption}

\usepackage{tikz}
\title{Tiered tree, Parking function and Postnikov-Shapiro algebra }
\author{Biswadeep Bagchi \\ \href{mailto:biswadeepbagchi430@gmail.com}{biswadeepbagchi430@gmail.com} 
   \and Srinibas Swain\\ \href{mailto:srinibas@iiitg.ac.in}{srinibas@iiitg.ac.in} }
\date{}
\begin{document}
\maketitle
\begin{abstract}
 Tiered trees were introduced as a combinatorial object for counting absolutely indecomposable representation of certain quivers and torus orbit of certain homogeneous variety. In this paper, we define a bijection between the set of parallelogram polyominoes  and graphical parking functions. Moreover, we defined the space $\mathcal{S}_{G}$ for complete tiered graphs and described tiered graphs in terms of Whitney's operations.   
\end{abstract}

\section{Introduction}
{\em Tiered trees} were defined in \cite{DGGE19} as trees on vertices labelled $\{1, \dots, n\}$ with an integer valued  tiering function {\bf t} on the set of vertices such that if two vertices $a$ and $b$ with $a < b$ are adjacent then ${\bf t}(a) < {\bf t}(b)$. Tiered trees are generalisation of {\em intransitive trees} {\cite{Post97}} with two tiers. These trees related to the spanning trees of some inversion graphs associated to a permutation. Tiered trees are connected to other combinatorial objects such as recurrent configuration and level statistic on {\em abelian sandpile model} {\cite{DSSS19}}. The origin of tiered trees can be derived from geometric counting problems \cite{GLV}:
\begin{itemize}
    \item Counting absolutely irreducible representation of the supernova quivers (quivers with long legs attached to vertices), and

    \item Counting certain torus orbit of partial flag variety of type A over finite field with trivial stabilizer.
\end{itemize}
Tiered trees can be described in terms of symmetric function and theta operators \cite{DIBRW}. These operators enumerates tiered trees in terms of inversion numbers. \\ 
Parking functions are sequence in $\mathbb N^{n}$ which arose in the study of hashing function \cite{KW}. The name of such sequence came in terms of parking for cars. Assume that $n $ parking spots labeled 0 through $n-1$ are arranged in order on a street. Drivers drive cars $A_{1}, \dots A_{n}$. Car $A_{k}$ prefers spot $b_{k}$
, which means that driver will drive until they reaches
that spot and park there if it is not occupied. If  preferred spot is occupied the the driver will continue driving until they come to an unoccupied spot and park the car in this spot. If preferred spot and all the
spots after it are occupied, then the driver cannot park. A parking function of length $n$ is sequence $(b_{1}, \dots , b_{n}) \in \mathbb N^{n}$  of parking preferences so that all the cars can park. For an equivalent definition see Definition ~\ref{Pf}(Section ~\ref{result}). 

Parking functions have appeared in various area of mathematics, more specifically in combinatorics such as hyperplane arrangements \cite{St} and diagonal harmonics \cite{HAG}. Parking functions are related to labelled trees. The number of parking functions of length $n$ is $(n+1)^{n-1}$. By Cayley's formula this is exactly the number of labelled trees on $n+1$ vertices. Moreover, there is an explicit bijection between parking functions and labelled trees. Let $b=(b_{1}, \dots, b_{n})$ be a parking function, then the \textit{sum} of $b$ is defined by $s(b)= \sum_{i \geqslant 1} b_{i} $. The \textit{reversed sum} of $b $ is defined by $rs(b)= {n\choose 2} -s(b)$. The \textit{reversed sum enumerator} is given by the expression $\sum_{b} q^{rs(b)}$, where $b$ ranges over all parking functions of length $n$. Kreweras \cite{K80} proved that the inversion number of labelled trees on $(n+1)$ vertices and  reversed sum of parking function of length $n$ are equal. \textit{Graphical parking functions} were first introduced by Postnikov and Shapiro \cite{PS}. These parking functions are connected to abelian sandpile models of graphs where it is called \textit{superstable configuration} (see \cite{PS}). Graphical parking function or $G$-parking functions depend on a choice of a graph $G$. For a set of $G$-parking functions, we have similar notions of sum $S(b)$ and reversed sum $rs(b)$, where $b$ is a $G$ -parking function. Postnikov and Shapiro proved that for a directed graph $G$ the number of $G$-parking function is equal to the number of spannig trees of $G$ by using Matrix-Tree theorem. Further, they defined the monomial and power ideals on labelled graphs. In \cite{P16} the authors extended the definition of $G$-parking function on any simple graphs and provided an bijection between $G $-parking function and a spannig trees of graphs by 
\[  \sum_{b} q^{rs(b)}= \sum_{T} q^{\kappa(G,T)}  , \] where $b$ ranges over all $G$-parking functions and $T$ ranges over all spannig tress of $G$. The number $\kappa(G,T)$ is a generalisation of the inversion number of spannig trees. When $G= K_{n+1}$, the complete graph on $(n+1)$ vertices the number of $G$-parking functions is same as number of classical parking functions of length $n$.\\
In the same paper \cite{PS}, Postnikov and Shapiro defined some algebras related to graphs. The algebras $\mathcal{B}_{G}$ and $\mathcal{C}_{G}$ which are called power algebra and counting spannig tree algebra of a graph\cite{PS} (see Section ~\ref{PS}). They showed that these algebras are isomorphic and it counts the number of spannig trees of a graph $G$. Moreover, the $k$-graded components of these algebras are equal to the the number of spannig trees with external activity $e(G)-n-k$ (see \cite{PS}).

In this paper, we define the graphical parking configuration on tiered trees and provided a bijection between set of parking configurations and labelled parallelogram polyominoes. The proof of the bijection is based on the algorithm defined over the unstable vertices (see Section~\ref{result}) of a graph. Later, we have described Postnikov-Shapiro algebra for tiered trees and forests respectively. We extended the definition of a \textit{dual graph} of a tiered graph with $m$ tiers and $n$ vertices, where $m \leqslant n$ and proved that tired graph and its dual are isomorphic.    

This paper is organised in the following way: In Section~\ref{def}, we introduce the combinatorial definitions surrounding graphs and tiered trees. In Section ~\ref{result}, we state and prove our main result of the paper regarding parallelogram polyomino and G-parking function. In Section ~\ref{PS} we recall the definitions of the algebras $\mathcal{B}_{G}$ and $\mathcal{C}_{G}$ and defined the space $\mathcal{S}_{G}$ for complete tiered graphs. We extended the notion of dual graphs as well as described tiered graphs in terms of Whitney's operations.

\section{Combinatorial definitions}\label{def}

\begin{Def}
In this paper, we assume $G$ to be a finite simple undirected  graph with the vertex set $V$ and edge set $E$. A rooted graph $G$ is a graph $(V,E)$ with a distinguished vertex $ r \in V$ designated as root. A tree is a connected graph with no cycles. A tree $T$ is a spanning tree of a graph $G$ if $T$ is a spanning subgraph of $G$. We denote $ST(G)$ be the set of all spannig trees of $G$. 
\end{Def}
\begin{Def}
    Let $T$ be a rooted tree with root  $ r \in $ V. We define the distance of a vertex $i$ from $r$, denoted by $dist(i,r)$ or simply $dist(i)$, is the shortest path between $i$ and $r$. The parent of a vertex $i\neq r$ is the unique neighbour $j$ of $i$ such that $dist(j) < dist(i)$ and we call $i$ is a child of $j$. We say $j$ is a descendant of $i$ if there exists a path $ i \xrightarrow{}i_{0} \xrightarrow[]{}   i_{1} \xrightarrow{} \dots \xrightarrow{ } i_{k}= j $. Note that,  we are not considering that $i$ is  descendant to itself. 
\end{Def}
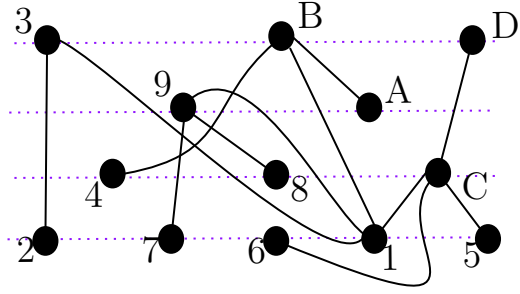
\begin{figure}
\centering
\tikzset{every picture/.style={line width=0.005pt}} 
\begin{tikzpicture}[x=0.45pt,y=0.45pt,yscale=-1,xscale=1]
\draw [color={rgb, 255:red, 144; green, 19; blue, 254 }  ,draw opacity=1 ][fill={rgb, 255:red, 144; green, 19; blue, 254 }  ,fill opacity=1 ][line width=0.75]  [dash pattern={on 0.84pt off 2.51pt}]  (113.79,44.66) -- (261.43,44.06) -- (522,43) ;
\draw [color={rgb, 255:red, 144; green, 19; blue, 254 }  ,draw opacity=1 ][fill={rgb, 255:red, 144; green, 19; blue, 254 }  ,fill opacity=1 ][line width=0.75]  [dash pattern={on 0.84pt off 2.51pt}]  (109.45,102.98) -- (257.09,102.38) -- (517.66,101.31) ;
\draw [color={rgb, 255:red, 144; green, 19; blue, 254 }  ,draw opacity=1 ][fill={rgb, 255:red, 144; green, 19; blue, 254 }  ,fill opacity=1 ][line width=0.75]  [dash pattern={on 0.84pt off 2.51pt}]  (113.79,157.97) -- (261.43,157.36) -- (522,156.3) ;
\draw [color={rgb, 255:red, 144; green, 19; blue, 254 }  ,draw opacity=1 ][fill={rgb, 255:red, 144; green, 19; blue, 254 }  ,fill opacity=1 ][line width=0.75]  [dash pattern={on 0.84pt off 2.51pt}]  (108,209.62) -- (255.64,209.02) -- (516.21,207.95) ;
\draw  [fill={rgb, 255:red, 0; green, 0; blue, 0 }  ,fill opacity=1 ][line width=0.75]  (131.16,42) .. controls (131.16,35.55) and (135.7,30.33) .. (141.29,30.33) .. controls (146.89,30.33) and (151.43,35.55) .. (151.43,42) .. controls (151.43,48.44) and (146.89,53.66) .. (141.29,53.66) .. controls (135.7,53.66) and (131.16,48.44) .. (131.16,42) -- cycle ;
\draw  [fill={rgb, 255:red, 0; green, 0; blue, 0 }  ,fill opacity=1 ][line width=0.75]  (328.03,38.66) .. controls (328.03,32.22) and (332.56,27) .. (338.16,27) .. controls (343.76,27) and (348.29,32.22) .. (348.29,38.66) .. controls (348.29,45.11) and (343.76,50.33) .. (338.16,50.33) .. controls (332.56,50.33) and (328.03,45.11) .. (328.03,38.66) -- cycle ;
\draw  [fill={rgb, 255:red, 0; green, 0; blue, 0 }  ,fill opacity=1 ][line width=0.75]  (488.71,42) .. controls (488.71,35.55) and (493.24,30.33) .. (498.84,30.33) .. controls (504.44,30.33) and (508.97,35.55) .. (508.97,42) .. controls (508.97,48.44) and (504.44,53.66) .. (498.84,53.66) .. controls (493.24,53.66) and (488.71,48.44) .. (488.71,42) -- cycle ;
\draw  [fill={rgb, 255:red, 0; green, 0; blue, 0 }  ,fill opacity=1 ][line width=0.75]  (245.52,98.65) .. controls (245.52,92.21) and (250.05,86.98) .. (255.65,86.98) .. controls (261.25,86.98) and (265.78,92.21) .. (265.78,98.65) .. controls (265.78,105.09) and (261.25,110.31) .. (255.65,110.31) .. controls (250.05,110.31) and (245.52,105.09) .. (245.52,98.65) -- cycle ;
\draw  [fill={rgb, 255:red, 0; green, 0; blue, 0 }  ,fill opacity=1 ][line width=0.75]  (401.85,98.65) .. controls (401.85,92.21) and (406.39,86.98) .. (411.99,86.98) .. controls (417.58,86.98) and (422.12,92.21) .. (422.12,98.65) .. controls (422.12,105.09) and (417.58,110.31) .. (411.99,110.31) .. controls (406.39,110.31) and (401.85,105.09) .. (401.85,98.65) -- cycle ;
\draw  [fill={rgb, 255:red, 0; green, 0; blue, 0 }  ,fill opacity=1 ][line width=0.75]  (459.76,153.63) .. controls (459.76,147.19) and (464.29,141.97) .. (469.89,141.97) .. controls (475.48,141.97) and (480.02,147.19) .. (480.02,153.63) .. controls (480.02,160.08) and (475.48,165.3) .. (469.89,165.3) .. controls (464.29,165.3) and (459.76,160.08) .. (459.76,153.63) -- cycle ;
\draw  [fill={rgb, 255:red, 0; green, 0; blue, 0 }  ,fill opacity=1 ][line width=0.75]  (323.69,155.3) .. controls (323.69,148.86) and (328.22,143.64) .. (333.82,143.64) .. controls (339.41,143.64) and (343.95,148.86) .. (343.95,155.3) .. controls (343.95,161.74) and (339.41,166.96) .. (333.82,166.96) .. controls (328.22,166.96) and (323.69,161.74) .. (323.69,155.3) -- cycle ;
\draw  [fill={rgb, 255:red, 0; green, 0; blue, 0 }  ,fill opacity=1 ][line width=0.75]  (186.16,154.33) .. controls (186.16,147.89) and (190.7,142.67) .. (196.3,142.67) .. controls (201.89,142.67) and (206.43,147.89) .. (206.43,154.33) .. controls (206.43,160.77) and (201.89,166) .. (196.3,166) .. controls (190.7,166) and (186.16,160.77) .. (186.16,154.33) -- cycle ;
\draw  [fill={rgb, 255:red, 0; green, 0; blue, 0 }  ,fill opacity=1 ][line width=0.75]  (501.73,209.32) .. controls (501.73,202.88) and (506.27,197.65) .. (511.86,197.65) .. controls (517.46,197.65) and (522,202.88) .. (522,209.32) .. controls (522,215.76) and (517.46,220.98) .. (511.86,220.98) .. controls (506.27,220.98) and (501.73,215.76) .. (501.73,209.32) -- cycle ;
\draw  [fill={rgb, 255:red, 0; green, 0; blue, 0 }  ,fill opacity=1 ][line width=0.75]  (406.19,209.32) .. controls (406.19,202.88) and (410.73,197.65) .. (416.33,197.65) .. controls (421.92,197.65) and (426.46,202.88) .. (426.46,209.32) .. controls (426.46,215.76) and (421.92,220.98) .. (416.33,220.98) .. controls (410.73,220.98) and (406.19,215.76) .. (406.19,209.32) -- cycle ;
\draw  [fill={rgb, 255:red, 0; green, 0; blue, 0 }  ,fill opacity=1 ][line width=0.75]  (323.68,210.98) .. controls (323.68,204.54) and (328.22,199.32) .. (333.81,199.32) .. controls (339.41,199.32) and (343.95,204.54) .. (343.95,210.98) .. controls (343.95,217.43) and (339.41,222.65) .. (333.81,222.65) .. controls (328.22,222.65) and (323.68,217.43) .. (323.68,210.98) -- cycle ;
\draw  [fill={rgb, 255:red, 0; green, 0; blue, 0 }  ,fill opacity=1 ][line width=0.75]  (235.38,209.32) .. controls (235.38,202.88) and (239.92,197.65) .. (245.51,197.65) .. controls (251.11,197.65) and (255.65,202.88) .. (255.65,209.32) .. controls (255.65,215.76) and (251.11,220.98) .. (245.51,220.98) .. controls (239.92,220.98) and (235.38,215.76) .. (235.38,209.32) -- cycle ;
\draw  [fill={rgb, 255:red, 0; green, 0; blue, 0 }  ,fill opacity=1 ][line width=0.75]  (129.71,210.98) .. controls (129.71,204.54) and (134.25,199.32) .. (139.84,199.32) .. controls (145.44,199.32) and (149.98,204.54) .. (149.98,210.98) .. controls (149.98,217.43) and (145.44,222.65) .. (139.84,222.65) .. controls (134.25,222.65) and (129.71,217.43) .. (129.71,210.98) -- cycle ;
\draw [line width=0.75]    (141.29,42) -- (139.84,210.98) ;
\draw [line width=0.75]    (141.29,42) .. controls (149.98,16.34) and (370.01,244.61) .. (406.19,209.32) ;
\draw [line width=0.75]    (345.4,48.66) -- (416.33,197.65) ;
\draw [line width=0.75]    (346.85,38.66) -- (411.99,98.65) ;
\draw [line width=0.75]    (206.43,154.33) .. controls (297.63,137.97) and (274.47,94.35) .. (332.37,44.36) ;
\draw [line width=0.75]    (255.65,98.65) -- (333.82,155.3) ;
\draw [line width=0.75]    (259.99,92.98) .. controls (317.9,43) and (387.38,207.95) .. (416.33,209.32) ;
\draw [line width=0.75]    (257.09,102.38) -- (245.51,209.32) ;
\draw [line width=0.75]    (469.89,153.63) -- (511.86,209.32) ;
\draw [line width=0.75]    (333.81,210.98) .. controls (555.29,302.93) and (411.99,203.62) .. (469.89,153.63) ;
\draw [line width=0.75]    (498.84,42) -- (469.89,153.63) ;
\draw [line width=0.75]    (459.76,153.63) -- (416.33,209.32) ;

\draw (113.88,206.07) node [anchor=north west][inner sep=0.75pt]  [font=\Large,rotate=-359.41] [align=left] {2};
\draw (307.88,207.07) node [anchor=north west][inner sep=0.75pt]  [font=\Large,rotate=-359.41] [align=left] {6};
\draw (218.88,206.07) node [anchor=north west][inner sep=0.75pt]  [font=\Large,rotate=-359.41] [align=left] {7};
\draw (419.88,213.07) node [anchor=north west][inner sep=0.75pt]  [font=\Large,rotate=-359.41] [align=left] {1};
\draw (488.86,210.98) node [anchor=north west][inner sep=0.75pt]  [font=\Large,rotate=-359.41] [align=left] {5};
\draw (487.88,152.07) node [anchor=north west][inner sep=0.75pt]  [font=\Large,rotate=-359.41] [align=left] {C};
\draw (343.95,155.3) node [anchor=north west][inner sep=0.75pt]  [font=\Large,rotate=-359.41] [align=left] {8};
\draw (170.88,160.07) node [anchor=north west][inner sep=0.75pt]  [font=\Large,rotate=-359.41] [align=left] {4};
\draw (422.88,71.07) node [anchor=north west][inner sep=0.75pt]  [font=\Large,rotate=-359.41] [align=left] {A};
\draw (228.88,67.07) node [anchor=north west][inner sep=0.75pt]  [font=\Large,rotate=-359.41] [align=left] {9};
\draw (248.88,87.07) node [anchor=north west][inner sep=0.75pt]  [font=\Large,rotate=-359.41] [align=left] {9};
\draw (512.88,17.07) node [anchor=north west][inner sep=0.75pt]  [font=\Large,rotate=-359.41] [align=left] {D};
\draw (349.88,8.07) node [anchor=north west][inner sep=0.75pt]  [font=\Large,rotate=-359.41] [align=left] {B};
\draw (111.88,14.07) node [anchor=north west][inner sep=0.75pt]  [font=\Large,rotate=-359.41] [align=left] {3};
\end{tikzpicture}
\caption{A tier tree with tier $(5,2,2,3)$.}\label{fig1}
\end{figure}
\begin{Def}
    A labelled {\em tiered graph } is a graph $(G,\textbf{t})$ where $ G =(V,E) $ with each vertex is labelled with $\{1,2, \dots, n\}$ with a surjective function ${\bf t:} V \mapsto [m]$ where $m \leqslant n$ subject to the following condition:

    \begin{itemize}
        \item For two vertices i and j, if $ (i,j) \in $E then $ {\bf t}(i) \neq {\bf t}(j)$. 
        \item if $(i,j) \in E$ and $i < j$  then $ {\bf t}(i) < {\bf t}(j)$ .
    \end{itemize}
\end{Def}
\begin{figure}
\centering
 \tikzset{every picture/.style={line width=0.75pt}} 

\begin{tikzpicture}[x=0.45pt,y=0.45pt,yscale=-1,xscale=1]

\draw [color={rgb, 255:red, 144; green, 19; blue, 254 }  ,draw opacity=1 ][fill={rgb, 255:red, 144; green, 19; blue, 254 }  ,fill opacity=1 ][line width=0.75]  [dash pattern={on 0.84pt off 2.51pt}]  (113.79,44.66) -- (261.43,44.06) -- (522,43) ;
\draw [color={rgb, 255:red, 144; green, 19; blue, 254 }  ,draw opacity=1 ][fill={rgb, 255:red, 144; green, 19; blue, 254 }  ,fill opacity=1 ][line width=0.75]  [dash pattern={on 0.84pt off 2.51pt}]  (113.79,157.97) -- (261.43,157.36) -- (522,156.3) ;
\draw  [fill={rgb, 255:red, 0; green, 0; blue, 0 }  ,fill opacity=1 ][line width=0.75]  (328.03,38.66) .. controls (328.03,32.22) and (332.56,27) .. (338.16,27) .. controls (343.76,27) and (348.29,32.22) .. (348.29,38.66) .. controls (348.29,45.11) and (343.76,50.33) .. (338.16,50.33) .. controls (332.56,50.33) and (328.03,45.11) .. (328.03,38.66) -- cycle ;
\draw  [fill={rgb, 255:red, 0; green, 0; blue, 0 }  ,fill opacity=1 ][line width=0.75]  (265.52,100.65) .. controls (265.52,94.21) and (270.05,88.98) .. (275.65,88.98) .. controls (281.25,88.98) and (285.78,94.21) .. (285.78,100.65) .. controls (285.78,107.09) and (281.25,112.31) .. (275.65,112.31) .. controls (270.05,112.31) and (265.52,107.09) .. (265.52,100.65) -- cycle ;
\draw  [fill={rgb, 255:red, 0; green, 0; blue, 0 }  ,fill opacity=1 ][line width=0.75]  (401.85,98.65) .. controls (401.85,92.21) and (406.39,86.98) .. (411.99,86.98) .. controls (417.58,86.98) and (422.12,92.21) .. (422.12,98.65) .. controls (422.12,105.09) and (417.58,110.31) .. (411.99,110.31) .. controls (406.39,110.31) and (401.85,105.09) .. (401.85,98.65) -- cycle ;
\draw  [fill={rgb, 255:red, 0; green, 0; blue, 0 }  ,fill opacity=1 ][line width=0.75]  (403.76,155.63) .. controls (403.76,149.19) and (408.29,143.97) .. (413.89,143.97) .. controls (419.48,143.97) and (424.02,149.19) .. (424.02,155.63) .. controls (424.02,162.08) and (419.48,167.3) .. (413.89,167.3) .. controls (408.29,167.3) and (403.76,162.08) .. (403.76,155.63) -- cycle ;
\draw  [fill={rgb, 255:red, 0; green, 0; blue, 0 }  ,fill opacity=1 ][line width=0.75]  (269.16,156.33) .. controls (269.16,149.89) and (273.7,144.67) .. (279.3,144.67) .. controls (284.89,144.67) and (289.43,149.89) .. (289.43,156.33) .. controls (289.43,162.77) and (284.89,168) .. (279.3,168) .. controls (273.7,168) and (269.16,162.77) .. (269.16,156.33) -- cycle ;
\draw  [fill={rgb, 255:red, 0; green, 0; blue, 0 }  ,fill opacity=1 ][line width=0.75]  (331.19,100.32) .. controls (331.19,93.88) and (335.73,88.65) .. (341.33,88.65) .. controls (346.92,88.65) and (351.46,93.88) .. (351.46,100.32) .. controls (351.46,106.76) and (346.92,111.98) .. (341.33,111.98) .. controls (335.73,111.98) and (331.19,106.76) .. (331.19,100.32) -- cycle ;
\draw [color={rgb, 255:red, 144; green, 19; blue, 254 }  ,draw opacity=1 ][fill={rgb, 255:red, 144; green, 19; blue, 254 }  ,fill opacity=1 ][line width=0.75]  [dash pattern={on 0.84pt off 2.51pt}]  (123.93,99.78) -- (271.57,99.18) -- (532.14,98.12) ;
\draw    (338.16,38.66) -- (280,96.6) ;
\draw    (338.16,38.66) -- (414,91.6) ;
\draw    (341.33,100.32) -- (408,150.6) ;
\draw  [dash pattern={on 4.5pt off 4.5pt}]  (341,46) -- (414,149.6) ;
\draw  [dash pattern={on 4.5pt off 4.5pt}]  (284,107) -- (410,162.6) ;
\draw  [dash pattern={on 4.5pt off 4.5pt}]  (413,103.6) -- (414,149.6) ;
\draw  [dash pattern={on 4.5pt off 4.5pt}]  (338.16,38.66) -- (283,148.6) ;
\draw    (276,107) -- (276,148.6) ;
\draw  [dash pattern={on 4.5pt off 4.5pt}]  (285,160) -- (405,102.6) ;
\draw    (338.16,38.66) -- (341.33,88.65) ;
\draw  [dash pattern={on 4.5pt off 4.5pt}]  (279.3,156.33) -- (338,105.6) ;

\draw (309,14) node [anchor=north west][inner sep=0.75pt]   [align=left] {{\Large 6}};
\draw (389,163) node [anchor=north west][inner sep=0.75pt]  [font=\Large] [align=left] {2};
\draw (250,163) node [anchor=north west][inner sep=0.75pt]   [align=left] {{\Large 1}};
\draw (353,85) node [anchor=north west][inner sep=0.75pt]   [align=left] {{\Large 4}};
\draw (424,81) node [anchor=north west][inner sep=0.75pt]   [align=left] {{\Large 5}};
\draw (246,85) node [anchor=north west][inner sep=0.75pt]   [align=left] {{\Large 3}};
\end{tikzpicture}
\caption{A tiered tree (black edges) and a complete tiered graph (black and dotted edges).}\label{fig2}
\end{figure}
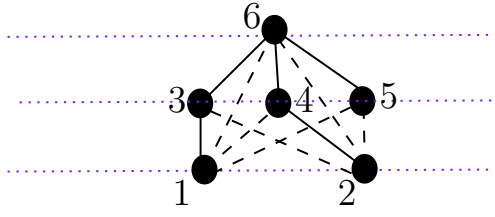

For the rest of this paper, we write tiered graph as $G$ instead of a tuple $(G,t)$. When tiered graph is a tree then we say that it is tiered tree with tiering function $\bf{t}$ and  we denote it by $T$. For example see Figure~\ref{fig1}. These trees  are the generalisation of Postnikov's intransitive trees {\cite{Post97}} on two tiers. We say that a tree is fully tiered if each tiers contain exactly one vertex. Tiered trees are related to the spanning trees of inversion graph associate to a permutation $\sigma$ where two vertices $i$ and $j$ are joined whenever for $i < j, \text{and} \hspace{0.5 mm}\sigma(i) > \sigma(j)$. A labelling of a tiered graph is said to be \textit{standard} if all the labels are from $\{1,2, \dots, n\}$. 

\begin{Def}
A tiered rooted tree is tiered tree with tiering function $\textbf{t}$ and a root $r$. 
\end{Def}

\begin{Def}
    Given a graph $G=(V,E)$, let $T$ be its spannig tree. For a total order $\prec$, we say that
    \begin{itemize}
        \item $e \in T$ is \textit{internally active } if it is the minimal edge according to $\prec$, in a cut of $G$ determined by the edge $e$ and $T$. 
        \item  $e \notin T$ is \textit{externally active} if it is the minimal edge according to $\prec$, in the unique cycle determined by $T \cup \{e\}$.
    \end{itemize}
\end{Def}
We denote $\bf{int}(T)$ be the number of internally active edges
and $\bf{ext}(T) $ be the number of externally active edges.   We define the Tutte polynomial of $G$ as, 
$$ T_{G}(x,y)= \sum_{T \in ST(G)} x^{ \bf int(T)} y^{\bf ext{(T)}} ,$$
 where order of edges are defined lexicographically.   
Note that from the classical result \cite{T54}, Tutte polynomial is independent of choice of ordering on the edges.  

\begin{Def}
An inversion of a tree $T$ is a pair of vertices $(i,j)$ such that $j$ is a descendant of $i$ and $i > j$. If $T$ is a spannnig tree of a rooted graph $G$ then the $\kappa$- inversion of $T$ is the pair $(i,j)$ such that $i$ is not a root and parent $p(i)$ of $i$ and $j$ are in the edge set of a rooted graph $G=(V,E)$ . The number of $\kappa$-inversion is denoted  by $inv_{\kappa}(T)$.

\end{Def}
An $\alpha$- tree is a tiered rooted tree where $\alpha = (\alpha_{1}, \dots, \alpha_{k})$ is a composition of $n$ such that $\sum_{i}\alpha_{i}=n $. As a generalization, it can be allowed to have repeated labelling on the vertices of the tiered tree (see \cite{DIBRW}).  A statistic on a class of combinatorial object, say $\mathcal{C}$ is a function from $\mathcal{C}$ to a set $\bf{S}$ (ideally taken as $\mathbb N$). A statistic on a spannig tree $T$ can be interpreted in terms of specialization of Tutte polynomial.   
\begin{Def}
For a graph G, a statistics $ \bf stat: ST(G) \mapsto \mathbb{N} $ is {\em Tutte descriptive } if
$$ T_{G}(1,q)= \sum_{T \in ST(G)} q^{\bf stat(T)}.  $$
\end{Def}
\begin{proposition}
    For any labelled tiered graph the statistic $ {\bf inv_{\kappa}}(T)$ is Tutte descriptive. 
\end{proposition}
Note that the number of inversion of a $\alpha$-tree is the same as its standardisation \cite{DIBRW}.

\begin{Def}
    For a tiered tree $T$, we say that $(i,j)$ is a \textit{compatible pair} if either $\textbf{t} (i) < \textbf{t}(j)$ and $i < j$ or $\textbf{t}(i) > \textbf{t}(j)$ and $i >j$
\end{Def}
For example, the tiered tree in Figure  ~\ref{fig2} the compatibel pairs (the dotted edges) are 
$$(1,6),(1,4),(1,5), (2,3),(2,5),(2,6).$$ 
Let $\textbf{TT}(\alpha)$ be the set of all $\alpha$ trees, we define its \textit{compatibility graph} $G_{t}$ by joining all compatible pairs. The $\alpha$-trees are the spannig trees of its compatibility graph. Note that when the labels of the tiered trees are standard labelled then compatibility graphs are the complete tiered graphs.  
In \cite{DIBRW} authors defined the generalisation of the inversion number of rooted trees to labelled tiered rooted trees. We recall the inversion number for tiered rooted trees is the pair of non rooted vertices $(i,j)$ such that
\begin{enumerate}
    \item [(1)] $j$ is a descendant of $i$,
    \item[(2)] $j $ compatible with the parent of $i$,
    \item[(3)] the pairs $(i,j)$ where $i > j$ or $i = j $
\end{enumerate}
When rooted tiered trees have standard labels then number of generalised inversion is just the $\kappa$-inversion of the spannig tree of a tiered graph. 
\section{Parking configuration and trees}\label{result}
In this section we state and prove the result regarding parallelogram  polynominoes and parking configuration. 
\begin{Def}
   Let $m$ and $n$ are two positive integers. A North-East path (lattice path ) $P$ in a $m \times n$  grid, is a sequence $P=(p_{0}, p_{1}, \dots, p_{n})$ with steps in $S= \{(0,1),(1,0)\}$ from $(0,0)$ to $(m,n)$. The $(0,1)$ steps are called north step and the $(1,0)$ step as east. We denote the north steps and east steps by $N^{'}$s and $E^{'}$s respectively. In Figure~\ref{fig3} the lattice path is encoded as $NNNEEENNENEENNE$. Let $\textbf{P}(m,n)$ be the set of lattice paths. 
\end{Def}
 \begin{Def}
     Let $P$ be a lattice path in an $m \times n$ grid. $P$ is a \textit{labelled path} if we add labels (from 1 to $n$) to the north steps of the path $P$. We place the labels in an increasing order in the north step whenever they lie in the same column. That is, labels are strictly increasing from bottom to top.  
 \end{Def}

  \begin{Def} 
   A \textit{parallelogram polynominoes} is a pair of lattice paths $(\textbf{r},\textbf{b})$ such that one path (say) $\bf{r}$ is always lie above the path $\bf{b}$ That is, if the lattice paths are represented by  sequences $\textbf{r}= x_{0}y_{1} \dots y_{m}x_{n}$ and $\textbf{b}= x^{\prime}_{0}y^{\prime}_{0} \dots x^{\prime}_{n}y^{\prime}_{m} $ then $x_{i}> x^{\prime}_{j}$ and $y_{i} > y^{\prime}_{j}$ for $i,j \in \{1,2, \dots m, m+1, \dots, n\}$. A \textit{labelled parallelogram polynominoes} is a parallelogram polynominoes where the positive integers are placed in a square of a grid containing a vertical step of a red path $\bf{r}$ and horizontal step of a blue path $\bf{b}$ such that the numbers appearing from bottom to top are strictly increasing and numbers from left to right are strictly decreasing.    
  
\end{Def}
We denote the set of labelled parallelogram polynominoes of size $m \times n$ by $\textbf{LPP}(m,n)$.  
We consider three types of labels; red, black and blue. The bottom left (most) square containing the black number which contains both vertical red paths and horizontal blue paths and rest are red numbers along vertical steps and blue numbers along horizontal steps. 
\begin{Def}[\cite{PS}]
    Let $G$ be a graph, we define the toppling operator $\Delta(i)= deg(i)\alpha(i)- \sum_{j} F(i,j) \alpha(j)$, where $F(i,j)$ is a function either takes the value 0 or 1 and $\alpha(i)$ be a vector in $\mathbb Z ^{n+1}$ such that the $i^{th}$ position takes the value 1 and other position 0. For a subset $A$ of $[n+1]$, we have $\Delta(A)= \sum_{j \in A} \Delta(j)$.  
\end{Def}
\begin{Def}
The abelian sandpile model (ASM) on a graph $G$ is a dynamic process in which there is a random walk on a \textit{configuration } $c$ ( a vector in $\mathbb Z^{n}_{+}$ with assigning a number $c(i)$ to each vertex $i$) by taking a vertex at random and performing toppling operations. In the sandpile model of $G$, we denote $s$ as a distinguished vertex called \textit{sink}. A configuration $c$ on $(G,s)$, we think the coordinate $c(i)$ as the \textit{grains of sand} at the vertex $i$. The vertex $i$ is \textit{unstable} if $c(i) \geqslant deg(i)$. we call a configuration stable if none of its vertices are unstable except the sink. An unstable vertex may be \textit{toppled} and gives a new configuration by giving one grain of sand to each of its neighbours. We denote the toppling process by $c \xrightarrow{i} c^{\prime}$.  A configuration is said to be \textit{recurrent} if $c(s)= deg(s)$ and there exits a sequence of vertices $i_{1},\dots i_{k}$ such that $c \xrightarrow{i_{1} }c^{\prime} \xrightarrow{i_{2}} \dots \xrightarrow{i_{k}} c$. Every unstable configuration can be made stable by applying sequence of toppling. We denote the set of all recurrent configuration by \textbf{Rec}($G$).  
\end{Def}
\begin{Def}
    Let $c$ be a configuration. We say that $c$ is \textit{non negative} if $c(i) \geqslant 0 $. We say that a vertex $j$ is unstable if the configurations $c$ and $c-\Delta(j)$ are both non negative and we can topple the vertex $j$.    
\end{Def} 
From the above definition we see that toppling of vertex $j$ depends on the new configuration $c^{\prime}=c-\Delta(j)$. 
Before stating our main theorem, we need a result from \cite{DSSS19}.

 

\begin{theorem}[Dukes, Selig, Smith, Steingrimsson \cite{DSSS19}]
Let $\pi \in S_{n}$  and $G$ be a permutation graph associated to $\pi$. Let $s \in [n]$. Then for every spannig tree $T$ of $G$ there is a combinatorial bijection between set of spannig trees of $G$ and recurrent configuration of $(G,s)$.
     
\end{theorem}
\begin{Def}\label{Pf}
    A parking function of size $n $ is function $f: [n] \mapsto [n]$ such that for the set $B=\{1 \leqslant j \leqslant n : f(j) \geqslant i\}$ we have $\lvert B \rvert \leqslant n+1-i$. 
    
\end{Def}
Parking functions have appeared various areas of combinatorics. Postnikov and Shapiro introduced graphical parking function or G-parking function (see definition in Section~\ref{PS} ). In this section instead of function, we use the term configuration. We denote $\textbf{PC}(G)$ be the set of all $G$-parking configuration.

\begin{Def}[\cite{PS}]
    A configuration $c$ is said to be graphical parking configuration or $G$- parking configuration if $c$ is non-negative and $c- \Delta(A)$ is not non negative for a subset $A$ of the set $\{1,2, \dots n+1\}$.
\end{Def}

\begin{Def}
    A coloring $c$ is a function from a class of object $\mathcal{X}$ to a set of colors $\mathcal{A}$. In this paper, the colors of vertices and paths are define by $c: \mathcal{X} \mapsto \{\text{red}, \text{blue}\}$. In this paper $\mathcal{X}$ is either $V$ the set of vertices or $P$ set of lattice paths in an $m \times n$ grid. 
\end{Def}

In the recent work \cite{DIBRW} , the authors defined a bijection between the labelled parallelogram polyomino and unique tiered tree with three tiers where the root of the tree in tier 2. Formally the result is stated below.
\begin{theorem}[\cite{DIBRW}]
There is a combinatorial bijection between set of labelled prallelogram polyominoes  $\textbf{LPP}(m+1,n+1)$ and set of rooted tiered trees $\textbf{TT}(m,1,n)$ with the root lies in tier 2.
\end{theorem}
\begin{Def}
    The \textit{area} of a labelled parallelogram polyomino is the number of cells between two paths that do not contain any label such that the labels appearing in the left are strictly greater than the labels below them.
\end{Def}
In this paper we establish a combinatorial bijection between labeled parallelogram polyomino and graphical parking configuration.
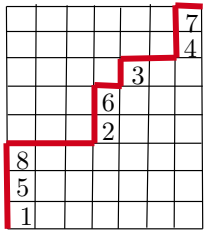
\begin{figure}
\centering
\tikzset{every picture/.style={line width=0.45pt}} 

\begin{tikzpicture}[x=0.45pt,y=0.45pt,yscale=-1,xscale=1]

\draw    (213,11) -- (214,198) ;
\draw    (377,11) -- (378,198) ;
\draw    (355,10) -- (354,197) ;
\draw    (334,9) -- (331,197) ;
\draw    (309,12) -- (307,198) ;
\draw    (287,10) -- (285,198) ;
\draw    (238,9) -- (237,198) ;
\draw    (264,10) -- (262,198) ;
\draw    (213,11) -- (377,11) ;
\draw    (213,32) -- (377,32) ;
\draw    (213,54) -- (377,54) ;
\draw    (214,78) -- (378,78) ;
\draw    (214,102) -- (378,102) ;
\draw    (213,126) -- (377,126) ;
\draw    (214,149) -- (378,149) ;
\draw    (214,175) -- (378,175) ;
\draw    (214,198) -- (378,198) ;
\draw [color={rgb, 255:red, 208; green, 2; blue, 27 }  ,draw opacity=1 ][line width=2.25]    (213,126) -- (214,198) ;
\draw [color={rgb, 255:red, 208; green, 2; blue, 27 }  ,draw opacity=1 ][line width=2.25]    (287,126) -- (213,126) ;
\draw [color={rgb, 255:red, 208; green, 2; blue, 27 }  ,draw opacity=1 ][line width=2.25]    (287,77) -- (287,126) ;
\draw [color={rgb, 255:red, 208; green, 2; blue, 27 }  ,draw opacity=1 ][line width=2.25]    (310,78) -- (287,77) ;
\draw [color={rgb, 255:red, 208; green, 2; blue, 27 }  ,draw opacity=1 ][line width=2.25]    (309,55) -- (309,78) ;
\draw [color={rgb, 255:red, 208; green, 2; blue, 27 }  ,draw opacity=1 ][line width=2.25]    (356,54) -- (309,55) ;
\draw [color={rgb, 255:red, 208; green, 2; blue, 27 }  ,draw opacity=1 ][line width=2.25]    (355,10) -- (356,54) ;
\draw [color={rgb, 255:red, 208; green, 2; blue, 27 }  ,draw opacity=1 ][line width=2.25]    (378,11) -- (355,10) ;

\draw (222,178) node [anchor=north west][inner sep=0.75pt]   [align=left] {1};
\draw (219,154) node [anchor=north west][inner sep=0.75pt]   [align=left] {5};
\draw (219,131) node [anchor=north west][inner sep=0.75pt]   [align=left] {8};
\draw (291,107) node [anchor=north west][inner sep=0.75pt]   [align=left] {2};
\draw (291,82) node [anchor=north west][inner sep=0.75pt]   [align=left] {6};
\draw (316,60) node [anchor=north west][inner sep=0.75pt]   [align=left] {3};
\draw (360,37) node [anchor=north west][inner sep=0.75pt]   [align=left] {4};
\draw (361,16) node [anchor=north west][inner sep=0.75pt]   [align=left] {7};
\end{tikzpicture}
\caption{A lattice path of size $7 \times 8$.}\label{fig3}
\end{figure}

\begin{theorem}\label{main}
    Let $U= (U_{1},U_{2},U_{3})$ be an ordered set partition of $[m+n-1]$. Then there exists a bijection between the set of labelled parallelogram polyominoes $\textbf{LPP}(U)$  and the G-parking configurations \textbf{PC}($G_{U}$), where $G_{U}$ is the compatible graph of the tiered tree $T$.
\end{theorem}

\begin{proof}
  Let us define a map $ \alpha: \textbf{LPP}(U) \rightarrow \textbf{PC}(G_{U}) $. We need to show that for any  path $P $ in
  $\textbf{LPP}(U) $, the image $ \alpha(P) \in \textbf{PC}(G_{U}) $. There exists a unique tiered tree $T \in \textbf{TT}(U)$ with three tiers with the root lies in tier 2 \cite{DIBRW}. As mentioned in Section ~\ref{def}, for a tiered tree there is a compatible graph which contains tiered tree as a spannig tree. Let $G_{U}$ be that graph. For each  pair of labels $(i,j)$ such that $ i < j$, the corresponding cells are colored white. Therefore, the white squares corresponds to the edges of $G_{U}$. That is, there is an edge $(i,j)$ where $i$ either blue or black and $j$ is either red or black. To show $\alpha(P)$ is parking configuration, we need to show $\alpha(P)-\Delta(I)$ is not non negative for a subset $I$ of $[m+n-1]$. Thus we need an ordering of labels to identify the sequence in which vertices are to be toppled. \\
\indent  Let us define a \textit{bounce path} $P$ starting from $(0,0)$ going east. When it meets a vertical step of a blue path it turns north. Again when it meets the horizontal step of a red path it turns east. Following this steps, the path $P$ ends at the $(m,n)$ coordinate. In order to get a toppling order, we project onto the horizontal steps of the bounce path which contains blue labels in its column. Similarly we project onto the vertical steps of the bounce path which contains red labels in its row. Thus we obtain a toppling order using $P$ denoted by $\textbf{TOP}(P)$. We start toppling the unstable vertices in the same order they appeared in $\textbf{TOP}(P)$. We disregard the black label in the toppling order which is the sink vertex (otherwise it will give a recurrent configuration).\\
\indent We now show the image $\alpha(P)$ is a $G$-parking configuration. Since $\alpha(P)$ is a configuration we assume that $\alpha(P)$ is stable. Let $j$ be the first vertex in our toppling order and $j \in I \subseteq [m+n-1]$. Define $\alpha(p)(j)$ be the number of white squares above (or to the right) of the squares containing $j$ when $j$ is blue (or red). Therefore in the toppling of $j$, it has as many grains as the squares above it (or right to it). Thus, $\alpha(P)(j)$ is a configuration with $j^{th}$ position is the number of white squares above it (or right to it). When we topple $j$, it loses $deg(j)$ number of grains and the neighbours of $j$ will receive one grain each. Let $A=\{1,2, \dots , m+n\}$ and $I$ be the subset of $[m+n-1]$. By our assumption, the set $A$ consists of three types of labels blue, red and black. The vertex  $j$ can give one grain to a neighbour which is not inside the set $I$. We write $deg_{I}(j)$ as the number of white squares above (or right) to $j$ such that there is an edge from the label $j \in I$ to a label from $A \setminus I$. Therefore, while toppling $j$ we have
\begin{equation}\label{eq}
    \alpha(P)(j) - deg(j) = \alpha(P)(j)- deg_{A}(j) +deg_{ I}(j).
\end{equation}
                                                    
Since $j$ contribute grains out side the set $I$ the right hand side of the Equation(~\ref{eq}) becomes $\alpha(P)(j) - deg_{A \setminus I}(j)$. Since $\alpha(P)(j)$ contains as many grains as the square above it, therefore we have
$$   \alpha(p)(j) -deg_{A \setminus I}(j) \geqslant deg_{A}(j). $$
Since $\alpha(P)$ is stable we have $\alpha(P)(j) \leqslant deg(j)-1$ for all $j$. This leads to a contradiction. Therefore, $\alpha(P)(j) -deg(j)$ is not non negative. After iterating this process for finitely many $j'$s we have the desired result. Hence the map is well defined.\\
Now we show the inverse image $\alpha^{-1}(P)$ exists. Let $\alpha(P)$ is a $G$-parking configuration. First we fix the sink vertex with black label placed at the bottom left corner. Let $c \in \textbf{PC}(G)$ is a stable configuration  and $c-\Delta(I)$ is non negative. We keep choosing unstable blue vertices from $c- \Delta(I)$. In this process we keep track the newly unstable vertices. We track blue vertices in decreasing order and red vertices in increasing order. We consider the subset $I$ consists of blue colour (or red) labels
 \begin{itemize}
     \item First fix the black level vertex to the left most corner in the bottom. 
     \item List all the unstable blue vertices in our toppling order from the set $I$ until it is empty. 
     \item List all the unstable red vertices from $A \setminus I$ until it is empty. 
 \end{itemize}
we repeat this process until all the unstable vertices are listed. We reconstruct the parallelogram polyomino from $\alpha^{-1}(P)$ as follows: the blue labels (or red) that became unstable in $c- \Delta(I)$, place them in the first row(or column) of the polyomino. Then for each blue (or red) vertex $j$, the vertices that become unstable upon its toppling will be the vertices that are in same column(or row) as $j$. Thus, after placing all the unstable vertices we will get polyomino. Hence $\alpha^{-1}(P)$ exists.  
\end{proof}
\begin{table}[h!]
  \begin{center}
     \begin{tabular}{|l|c|c|r|} 
     \hline
     & \textbf{Black} &  \textbf{Red} & \textbf{Blue}\\
           \hline
      $j$ & 8 &11 7 10 12 &3 6 5 4 1 9 2\\
      \hline
     $\alpha(p)(j)$ & 9 & 7 \hspace{.2mm}  6 \hspace{.2mm} 6  4&4 2 3 3 2 1 1 \\
     \hline
    \end{tabular}
  \end{center}
     \caption{Intial configuration of the polyominoes in Figure ~\ref{fig4}}\label{t1}
    \label{tab:table1}
\end{table}
Note that, instead of one coloured label, we may take the set $I$ with labels consisting both blue or red labels which we disregard as from the definition of $deg_{I}(i)$ defined in the beginning of the proof.
\begin{Example}
In Figure ~\ref{fig4}, for the path $P$ we have the initial configuration $$\alpha(P)= (9,7,6,6,4,4,2,3,3,2,1,1).$$ The bounce path is shown as a dotted black line. We determine the stability of the vertices (excluding the vertex $8$, which is the sink vertex) by using the toppling operator $\Delta(j)$ for a vertex $j$. For instance, $\alpha(P)- \Delta(3)= (9,7,1,6,4,4,3,2,3,3,3,1,1)$ which is non negative. So the vertex $3$ is unstable. Again $\alpha(P)-\Delta(6)= (9,7,6,6,4,1,2,3,3,2,0,1)$. So the vertex $6$ is also unstable. In this case we begin the process by toppling the vertex $3$. The vertex $3$ will loose $deg(3)$ number of grains, results instability of the vertices $7$ and $10$. Therefore, the blue label $3$ will be in the first row and red labels will be in it's column. Then for the instability of the vertex $6$, the vertex $11$ will be unstable. Therefore, the blue label $6$ will be placed in the second row and red label $11$ will be placed in corresponding column and so on.

\end{Example}
\begin{figure}
\centering   
\tikzset{every picture/.style={line width=0.75pt}} 

\begin{tikzpicture}[x=0.60pt,y=0.60pt,yscale=-1,xscale=1]

\draw [line width=0.75]    (152,31) -- (154,181.6) ;
\draw [color={rgb, 255:red, 0; green, 0; blue, 0 }  ,draw opacity=1 ]   (241,30) -- (243,180.6) ;
\draw [color={rgb, 255:red, 0; green, 0; blue, 0 }  ,draw opacity=1 ]   (211,31) -- (213,181.6) ;
\draw [line width=0.75]    (182,31) -- (184,181.6) ;
\draw [color={rgb, 255:red, 0; green, 0; blue, 0 }  ,draw opacity=1 ]   (307,29.8) -- (307,179.8) ;
\draw [color={rgb, 255:red, 0; green, 0; blue, 0 }  ,draw opacity=1 ]   (273,30) -- (275,180.6) ;
\draw    (370,28) -- (371,180.8) ;
\draw    (341,30.8) -- (341,179.6) ;
\draw    (154,180.6) -- (403,178.8) ;
\draw    (152,62) -- (404,58.8) ;
\draw    (151,90) -- (402,88.8) ;
\draw    (153,120) -- (256.03,118.38) -- (403,118.8) ;
\draw [color={rgb, 255:red, 0; green, 0; blue, 0 }  ,draw opacity=1 ]   (153,152) -- (404,148.8) ;
\draw [line width=0.75]    (402,27.2) -- (404,178.8) ;
\draw [line width=0.75]    (193,179.6) -- (154,180.6) ;
\draw    (151,32) -- (403,28.8) ;
\draw  [fill={rgb, 255:red, 155; green, 155; blue, 155 }  ,fill opacity=1 ] (152,32) -- (183,32) -- (183,60.8) -- (152,60.8) -- cycle ;
\draw  [fill={rgb, 255:red, 155; green, 155; blue, 155 }  ,fill opacity=1 ] (211,32) -- (242,32) -- (242,60.8) -- (211,60.8) -- cycle ;
\draw  [fill={rgb, 255:red, 155; green, 155; blue, 155 }  ,fill opacity=1 ] (212,151.8) -- (243,151.8) -- (243,180.6) -- (212,180.6) -- cycle ;
\draw [color={rgb, 255:red, 208; green, 2; blue, 27 }  ,draw opacity=1 ][line width=2.25]    (153,120) -- (154,180.6) ;
\draw [color={rgb, 255:red, 208; green, 2; blue, 27 }  ,draw opacity=1 ][line width=2.25]    (183,60.8) -- (184,121.4) ;
\draw [color={rgb, 255:red, 208; green, 2; blue, 27 }  ,draw opacity=1 ][line width=2.25]    (184,119.4) -- (153,120) ;
\draw [color={rgb, 255:red, 208; green, 2; blue, 27 }  ,draw opacity=1 ][line width=2.25]    (243,60.8) -- (184.5,62.3) ;
\draw [color={rgb, 255:red, 208; green, 2; blue, 27 }  ,draw opacity=1 ][line width=2.25]    (242,31) -- (243,60.8) ;
\draw [color={rgb, 255:red, 208; green, 2; blue, 27 }  ,draw opacity=1 ][line width=2.25]    (403,28.8) -- (242,31) ;
\draw [color={rgb, 255:red, 74; green, 144; blue, 226 }  ,draw opacity=1 ][line width=2.25]    (154,180.6) -- (212,180.6) ;
\draw [color={rgb, 255:red, 74; green, 144; blue, 226 }  ,draw opacity=1 ][line width=2.25]    (212,151.8) -- (213,179.6) ;
\draw [color={rgb, 255:red, 74; green, 144; blue, 226 }  ,draw opacity=1 ][line width=2.25]    (214,151.8) -- (308,148.8) ;
\draw [color={rgb, 255:red, 74; green, 144; blue, 226 }  ,draw opacity=1 ][line width=2.25]    (308,115.8) -- (308,148.8) ;
\draw [color={rgb, 255:red, 74; green, 144; blue, 226 }  ,draw opacity=1 ][line width=2.25]    (303,118.8) -- (344,117.8) ;
\draw [color={rgb, 255:red, 74; green, 144; blue, 226 }  ,draw opacity=1 ][line width=2.25]    (341,89) -- (342,116.8) ;
\draw [color={rgb, 255:red, 74; green, 144; blue, 226 }  ,draw opacity=1 ][line width=2.25]    (343,89.8) -- (402,88.8) ;
\draw [color={rgb, 255:red, 74; green, 144; blue, 226 }  ,draw opacity=1 ][line width=2.25]    (403,28.8) -- (403,88.8) ;
\draw [line width=2.25]  [dash pattern={on 2.53pt off 3.02pt}]  (154,152.6) -- (212,151.6) ;
\draw [line width=2.25]  [dash pattern={on 2.53pt off 3.02pt}]  (212,151.8) -- (211,60.8) ;
\draw [line width=2.25]  [dash pattern={on 2.53pt off 3.02pt}]  (213,61.3) -- (403,58.8) ;
\draw [line width=2.25]  [dash pattern={on 2.53pt off 3.02pt}]  (403,58.8) -- (403,28.8) ;
\draw [color={rgb, 255:red, 155; green, 155; blue, 155 }  ,draw opacity=1 ][line width=2.25]  [dash pattern={on 2.53pt off 3.02pt}]  (176,154) -- (183,165.8) ;
\draw [color={rgb, 255:red, 155; green, 155; blue, 155 }  ,draw opacity=1 ][line width=2.25]  [dash pattern={on 2.53pt off 3.02pt}]  (169,153.8) -- (186,181.6) ;
\draw [color={rgb, 255:red, 155; green, 155; blue, 155 }  ,draw opacity=1 ][line width=2.25]  [dash pattern={on 2.53pt off 3.02pt}]  (165,179.8) -- (156,165.8) ;
\draw [color={rgb, 255:red, 155; green, 155; blue, 155 }  ,draw opacity=1 ][line width=2.25]  [dash pattern={on 2.53pt off 3.02pt}]  (167,155) -- (158,166.8) ;
\draw [color={rgb, 255:red, 155; green, 155; blue, 155 }  ,draw opacity=1 ][line width=2.25]  [dash pattern={on 2.53pt off 3.02pt}]  (168,180) -- (179,172.8) ;

\draw (162,160) node [anchor=north west][inner sep=0.75pt]   [align=left] {{\large 8}};
\draw (192,157) node [anchor=north west][inner sep=0.75pt]  [color={rgb, 255:red, 74; green, 144; blue, 226 }  ,opacity=1 ] [align=left] {{\large 3}};
\draw (221,126) node [anchor=north west][inner sep=0.75pt]  [color={rgb, 255:red, 74; green, 144; blue, 226 }  ,opacity=1 ] [align=left] {{\large 6}};
\draw (253,127) node [anchor=north west][inner sep=0.75pt]  [color={rgb, 255:red, 74; green, 144; blue, 226 }  ,opacity=1 ] [align=left] {{\large 5}};
\draw (284,126) node [anchor=north west][inner sep=0.75pt]  [color={rgb, 255:red, 74; green, 144; blue, 226 }  ,opacity=1 ] [align=left] {{\large 4}};
\draw (159,128) node [anchor=north west][inner sep=0.75pt]  [color={rgb, 255:red, 208; green, 2; blue, 27 }  ,opacity=1 ] [align=left] {{\large 11}};
\draw (190,98) node [anchor=north west][inner sep=0.75pt]  [color={rgb, 255:red, 208; green, 2; blue, 27 }  ,opacity=1 ] [align=left] {{\large 7}};
\draw (185.5,68.3) node [anchor=north west][inner sep=0.75pt]  [color={rgb, 255:red, 208; green, 2; blue, 27 }  ,opacity=1 ] [align=left] {{\large 10}};
\draw (318,95) node [anchor=north west][inner sep=0.75pt]  [color={rgb, 255:red, 74; green, 144; blue, 226 }  ,opacity=1 ] [align=left] {{\large 1}};
\draw (350,68) node [anchor=north west][inner sep=0.75pt]  [color={rgb, 255:red, 74; green, 144; blue, 226 }  ,opacity=1 ] [align=left] {{\large 9}};
\draw (381,65) node [anchor=north west][inner sep=0.75pt]  [color={rgb, 255:red, 74; green, 144; blue, 226 }  ,opacity=1 ] [align=left] {{\large 2}\\};
\draw (247,38) node [anchor=north west][inner sep=0.75pt]  [color={rgb, 255:red, 208; green, 2; blue, 27 }  ,opacity=1 ] [align=left] {{\large 12}};

\end{tikzpicture}

\caption{ A labelled parallelogram polynomino}\label{fig4}
\end{figure}
\vspace{-.7cm}
\section{Tiered trees and P-S algebra}\label{PS}
\subsection{The PS algebra}
 In this section we define the $G$-parking function for a tiered graph $G$. The notations, definitions and results which we have discussed in this section, are defined in \cite{PS}. Postnikov and Shapiro defined two graded algebras on directed graphs and showed that their Hilbert series is equal to the number of spannig trees of the corresponding directed graph. Here we define the general constructions of $G$-parking function for rooted tiered trees. Let $G= (V,E)$ be a graph with vertices labelled with $ 1,2, \dots, n$. For a subset $I$ of $\{ 1,2, \dots ,n\}$ and $i \in I$ let
$$ d_{I}(i)= \sum_{k \notin I} a_{ik} ,  $$
which is the number of  edges that connects a vertex from $I$ to vertex outside of $I$. The $G$-parking function is the sequence $(b_{1}, \dots, b_{n}  ) $ where there exists a subset $I $ of $\{1,2, \dots, n\}$ for which $b_{i} < d_{I}(i)$. Let $\mathbb K$ be a field and $\mathcal{I}$ be a monomial ideal in $\mathbb K[x_{1}, \dots, x_{n}]$  generated by the monomials
$$ m_{I} = \prod_{i \in I} x_{i}^{d_{I}(i)},  $$
where $I$ is a subset of  $\{1,2, \dots, n\}$. Define the algebra $\mathcal{A}_{G}$ as the quotient $\mathcal{A}_{G}= \mathbb K[x_{1}, \dots, x_{n}]\big/{\mathcal{I} }$. 

We now recall the definition of  the power algebra $\mathcal{B}_{G}$ whose quotient is modulo $\mathcal{J} $, where $\mathcal{J}$ is the ideal in the polynomial ring 
$ {\mathbb K[x_{1}, \dots, x_{n}]}$ spanned by the monomials,
$$  p_{I}= ( \sum_{i \in I}x_{i} )^{D_{I}+1} ,$$
where $D_{I}$ is the total number of edges that joins a vertex $i$ from $I$ to a vertex outside $I$ that is, $V(G) \setminus I$. The square free algebras associated to a vertex labelled graph was introduced in \cite{PS}. For a field $\mathbb K$ of characteristic zero, let $\Phi_{G}$ be a graded algebra generated by the commuting variables $\phi_{e}$ for all edges $e \in G$ subject to the following relation
\[ (\phi_{e})^{2} =0  \hspace{1cm} \text{for any edge } e \in G,  
  \]  
  \[\prod_{e \in H} \phi_{e} =o \hspace{1cm} \text{for any nonslim subgraph} \hspace{0.1cm}H  \subset G,
\]
where a \textit{slim} subgraph $H$ of $G$ is a subgraph such that the complement subgraph $G \setminus H $ is connected. Postnikov and Shapiro defined the algebra $\mathcal{C}_{G}$ associated to spannig trees  which is a subalgebra of $\Phi_{G}$ spanned by the elements $$ X_{i}= \sum_{e \in G} c_{i,e} \phi_{e}, $$
where \[  c_{i,e} = \begin{cases}
    1 & \mbox{if e=(i,j) , i } < \mbox{ j }, \\
    -1 & \mbox{if e=(i,j), i   } > \mbox{j}, \\
    0 & \mbox{otherwise}. 
\end{cases}
\]
 The algebra $\mathcal{C}_{G}$ is called the spannig tree counting algebra of a graph $G$.
\begin{Example}
For the tiered graph $G$ in Figure~\ref{fig2}, the algebra $\mathcal{C}_{G}$ is generated by $\{X_{1}, X_{2}, X_{3}, X_{4}, X_{5}, X_{6}\}$ where $X_{i}$'s are given by the following relations
\begin{align*}
    X_{1} &= \phi_{13}+\phi_{14}+\phi_{15}+\phi_{16} ,~~~
    X_{2} = \phi_{23}+\phi_{24}+\phi_{25}+\phi_{26},~~~
    X_{3}= - \phi_{31}-\phi_{32}+\phi_{36} ,\\
    X_{4}& =  -\phi_{41}-\phi_{42}+\phi_{46},~~~
    X_{5} =  -\phi_{51}-\phi_{52}+\phi_{56},~~~
    X_{6}= -\phi_{61}-\phi_{62}-\phi_{63}-\phi_{64}-\phi_{65}.
\end{align*}
\end{Example}
The following theorem is one of the main results of \cite{PS}, which shows that why the algebra $\mathcal{C}_{G}$ is called algebra counting spannig trees of a graph $G$. 
\begin{theorem}[Postnikov and Shapiro \cite{PS}]
The dimension of the algebra $\mathcal{C}_{G}$ over the field $\mathbb K$ is equal to the the number of spannig trees of $G$. Moreover, the dimension of the $ k$-graded component  $\mathcal{C}_{G} (k)$ of the algebra $\mathcal{C}_{G}$ is the number of spannig trees with external activity $e(G) -n-k$ where $e(G)$ is the number of edges of $G$.     
\end{theorem}
From the above theorem we have the following result regarding Hilbert polynomial of $\mathcal{C}_{G}$ as a specialization of the Tutte polynomial of $G$.

\begin{corollary}
    Let $G$ be a graph, the Hilbert polynomial  $\mathcal{H}_{\mathcal{C}_{G}}(s)$ of the algebra $\mathcal{C}_{G}$ is given by

    $$ \mathcal{H}_{\mathcal{C}_{G}}(s) = T_{G}(1+s,s^{-1}).s^{e(G)-v(G)+c(G)} ,  $$
where $e(G) $ and $v(G)$ denote the number of edges and vertices of $G$ respectively and $c(G)$ denote the number of connected components.
\end{corollary}

  In this paper, we define the space $\mathcal{S}_{G}$ \cite{PS} for a tiered graph $G$. It is proved in (\cite{DGGE19}, Theorem 2.9 ) that the externally active edges in a complete tiered graph $G$ with $T$ as a spannig tree, are the edges $e_{r,j}$ where the vertex $r$ is the lowest label vertex according to the ordering $1 <2<3< \dots<n$ and $j \in G$ with $r < j  $ and $t(r) < t(j)$. Moreover, the vertex $j$ is adjacent to vertex $i$ in $T$ which was connected to $r$ before deleting it such that $j < i$ and $\textbf{t}(j)< \textbf{t}(i) $.  The space $\mathcal{S}_{G}$ in $\mathbb{K}[z_{1}, \dots, z_{n}]$ associated with a tiered graph $G$ is spanned by the elements 

$$   Z_{H}= \prod_{e \in G}   z_{e}, $$ where $z_{e} = z_{i}- z_{j}$ with $i < j$ and $\textbf{t}(i)< \textbf{t}(j)$. Let $\textbf{act}(G)$ be the set of externally active edges of $G$. We denote $T_{\textbf{act}}$ be a tiered graph obtained from a tiered tree by joining all the external active edges. 
\begin{figure}
    \centering
\tikzset{every picture/.style={line width=0.15pt}} 
\begin{tikzpicture}[x=0.45pt,y=0.45pt,yscale=-1,xscale=1]

\draw [color={rgb, 255:red, 144; green, 19; blue, 254 }  ,draw opacity=1 ][fill={rgb, 255:red, 144; green, 19; blue, 254 }  ,fill opacity=1 ][line width=0.75]  [dash pattern={on 0.84pt off 2.51pt}]  (113.79,44.66) -- (261.43,44.06) -- (522,43) ;
\draw [color={rgb, 255:red, 144; green, 19; blue, 254 }  ,draw opacity=1 ][fill={rgb, 255:red, 144; green, 19; blue, 254 }  ,fill opacity=1 ][line width=0.75]  [dash pattern={on 0.84pt off 2.51pt}]  (109.45,102.98) -- (257.09,102.38) -- (517.66,101.31) ;
\draw [color={rgb, 255:red, 144; green, 19; blue, 254 }  ,draw opacity=1 ][fill={rgb, 255:red, 144; green, 19; blue, 254 }  ,fill opacity=1 ][line width=0.75]  [dash pattern={on 0.84pt off 2.51pt}]  (113.79,157.97) -- (261.43,157.36) -- (522,156.3) ;
\draw [color={rgb, 255:red, 144; green, 19; blue, 254 }  ,draw opacity=1 ][fill={rgb, 255:red, 144; green, 19; blue, 254 }  ,fill opacity=1 ][line width=0.75]  [dash pattern={on 0.84pt off 2.51pt}]  (108,209.62) -- (255.64,209.02) -- (516.21,207.95) ;
\draw  [fill={rgb, 255:red, 0; green, 0; blue, 0 }  ,fill opacity=1 ][line width=0.75]  (131.16,42) .. controls (131.16,35.55) and (135.7,30.33) .. (141.29,30.33) .. controls (146.89,30.33) and (151.43,35.55) .. (151.43,42) .. controls (151.43,48.44) and (146.89,53.66) .. (141.29,53.66) .. controls (135.7,53.66) and (131.16,48.44) .. (131.16,42) -- cycle ;
\draw  [fill={rgb, 255:red, 0; green, 0; blue, 0 }  ,fill opacity=1 ][line width=0.75]  (187.61,44.36) .. controls (187.61,37.92) and (192.15,32.7) .. (197.74,32.7) .. controls (203.34,32.7) and (207.88,37.92) .. (207.88,44.36) .. controls (207.88,50.8) and (203.34,56.02) .. (197.74,56.02) .. controls (192.15,56.02) and (187.61,50.8) .. (187.61,44.36) -- cycle ;
\draw  [fill={rgb, 255:red, 0; green, 0; blue, 0 }  ,fill opacity=1 ][line width=0.75]  (429.71,42) .. controls (429.71,35.55) and (434.24,30.33) .. (439.84,30.33) .. controls (445.44,30.33) and (449.97,35.55) .. (449.97,42) .. controls (449.97,48.44) and (445.44,53.66) .. (439.84,53.66) .. controls (434.24,53.66) and (429.71,48.44) .. (429.71,42) -- cycle ;
\draw  [fill={rgb, 255:red, 0; green, 0; blue, 0 }  ,fill opacity=1 ][line width=0.75]  (245.52,98.65) .. controls (245.52,92.21) and (250.05,86.98) .. (255.65,86.98) .. controls (261.25,86.98) and (265.78,92.21) .. (265.78,98.65) .. controls (265.78,105.09) and (261.25,110.31) .. (255.65,110.31) .. controls (250.05,110.31) and (245.52,105.09) .. (245.52,98.65) -- cycle ;
\draw  [fill={rgb, 255:red, 0; green, 0; blue, 0 }  ,fill opacity=1 ][line width=0.75]  (331.87,101.6) .. controls (331.87,95.16) and (336.4,89.94) .. (342,89.94) .. controls (347.6,89.94) and (352.13,95.16) .. (352.13,101.6) .. controls (352.13,108.04) and (347.6,113.26) .. (342,113.26) .. controls (336.4,113.26) and (331.87,108.04) .. (331.87,101.6) -- cycle ;
\draw  [fill={rgb, 255:red, 0; green, 0; blue, 0 }  ,fill opacity=1 ][line width=0.75]  (428.87,155.26) .. controls (428.87,148.82) and (433.4,143.6) .. (439,143.6) .. controls (444.6,143.6) and (449.13,148.82) .. (449.13,155.26) .. controls (449.13,161.71) and (444.6,166.93) .. (439,166.93) .. controls (433.4,166.93) and (428.87,161.71) .. (428.87,155.26) -- cycle ;
\draw  [fill={rgb, 255:red, 0; green, 0; blue, 0 }  ,fill opacity=1 ][line width=0.75]  (323.69,155.3) .. controls (323.69,148.86) and (328.22,143.64) .. (333.82,143.64) .. controls (339.41,143.64) and (343.95,148.86) .. (343.95,155.3) .. controls (343.95,161.74) and (339.41,166.96) .. (333.82,166.96) .. controls (328.22,166.96) and (323.69,161.74) .. (323.69,155.3) -- cycle ;
\draw  [fill={rgb, 255:red, 0; green, 0; blue, 0 }  ,fill opacity=1 ][line width=0.75]  (186.16,154.33) .. controls (186.16,147.89) and (190.7,142.67) .. (196.3,142.67) .. controls (201.89,142.67) and (206.43,147.89) .. (206.43,154.33) .. controls (206.43,160.77) and (201.89,166) .. (196.3,166) .. controls (190.7,166) and (186.16,160.77) .. (186.16,154.33) -- cycle ;
\draw  [fill={rgb, 255:red, 0; green, 0; blue, 0 }  ,fill opacity=1 ][line width=0.75]  (474.87,206.6) .. controls (474.87,200.16) and (479.4,194.94) .. (485,194.94) .. controls (490.6,194.94) and (495.13,200.16) .. (495.13,206.6) .. controls (495.13,213.04) and (490.6,218.26) .. (485,218.26) .. controls (479.4,218.26) and (474.87,213.04) .. (474.87,206.6) -- cycle ;
\draw  [fill={rgb, 255:red, 0; green, 0; blue, 0 }  ,fill opacity=1 ][line width=0.75]  (385.93,208.48) .. controls (385.93,202.04) and (390.46,196.82) .. (396.06,196.82) .. controls (401.66,196.82) and (406.19,202.04) .. (406.19,208.48) .. controls (406.19,214.93) and (401.66,220.15) .. (396.06,220.15) .. controls (390.46,220.15) and (385.93,214.93) .. (385.93,208.48) -- cycle ;
\draw  [fill={rgb, 255:red, 0; green, 0; blue, 0 }  ,fill opacity=1 ][line width=0.75]  (235.38,209.32) .. controls (235.38,202.88) and (239.92,197.65) .. (245.51,197.65) .. controls (251.11,197.65) and (255.65,202.88) .. (255.65,209.32) .. controls (255.65,215.76) and (251.11,220.98) .. (245.51,220.98) .. controls (239.92,220.98) and (235.38,215.76) .. (235.38,209.32) -- cycle ;
\draw  [fill={rgb, 255:red, 0; green, 0; blue, 0 }  ,fill opacity=1 ][line width=0.75]  (129.71,210.98) .. controls (129.71,204.54) and (134.25,199.32) .. (139.84,199.32) .. controls (145.44,199.32) and (149.98,204.54) .. (149.98,210.98) .. controls (149.98,217.43) and (145.44,222.65) .. (139.84,222.65) .. controls (134.25,222.65) and (129.71,217.43) .. (129.71,210.98) -- cycle ;
\draw [line width=0.75]    (141.29,42) -- (139.84,210.98) ;
\draw [line width=0.75]    (342,101.6) -- (333.82,155.3) ;
\draw [line width=0.75]    (342,101.6) -- (250,202.6) ;
\draw    (197.74,56.02) -- (196.3,154.33) ;
\draw  [dash pattern={on 4.5pt off 4.5pt}]  (257.09,102.38) -- (234.25,121.9) -- (196.3,154.33) ;
\draw    (197.74,44.36) -- (254,92.6) ;
\draw  [dash pattern={on 4.5pt off 4.5pt}]  (333.82,155.3) -- (297,180.6) -- (255.65,209.32) ;
\draw    (439.84,42) -- (439,143.6) ;
\draw    (439,155.26) -- (402,200.6) ;
\draw    (440,159.32) -- (486,203.6) ;
\draw  [dash pattern={on 4.5pt off 4.5pt}]  (433.84,39) -- (396.06,208.48) ;
\draw  [dash pattern={on 4.5pt off 4.5pt}]  (443.97,39) -- (491,200.6) ;
\draw  [fill={rgb, 255:red, 0; green, 0; blue, 0 }  ,fill opacity=1 ][line width=0.75]  (310.38,206.32) .. controls (310.38,199.88) and (314.92,194.65) .. (320.51,194.65) .. controls (326.11,194.65) and (330.65,199.88) .. (330.65,206.32) .. controls (330.65,212.76) and (326.11,217.98) .. (320.51,217.98) .. controls (314.92,217.98) and (310.38,212.76) .. (310.38,206.32) -- cycle ;

\draw (109.88,206.07) node [anchor=north west][inner sep=0.75pt]  [font=\Large,rotate=-359.41] [align=left] {2};
\draw (368.88,207.07) node [anchor=north west][inner sep=0.75pt]  [font=\Large,rotate=-359.41] [align=left] {6};
\draw (214.88,206.07) node [anchor=north west][inner sep=0.75pt]  [font=\Large,rotate=-359.41] [align=left] {7};
\draw (442.88,120.07) node [anchor=north west][inner sep=0.75pt]  [font=\Large,rotate=-359.41] [align=left] {C};
\draw (488.86,210.98) node [anchor=north west][inner sep=0.75pt]  [font=\Large,rotate=-359.41] [align=left] {5};
\draw (343.95,155.3) node [anchor=north west][inner sep=0.75pt]  [font=\Large,rotate=-359.41] [align=left] {8};
\draw (170.88,160.07) node [anchor=north west][inner sep=0.75pt]  [font=\Large,rotate=-359.41] [align=left] {4};
\draw (266.88,67.07) node [anchor=north west][inner sep=0.75pt]  [font=\Large,rotate=-359.41] [align=left] {A};
\draw (351.88,78.07) node [anchor=north west][inner sep=0.75pt]  [font=\Large,rotate=-359.41] [align=left] {9};
\draw (248.88,87.07) node [anchor=north west][inner sep=0.75pt]  [font=\Large,rotate=-359.41] [align=left] {9};
\draw (248.88,87.07) node [anchor=north west][inner sep=0.75pt]  [font=\Large,rotate=-359.41] [align=left] {9};
\draw (295.88,204.07) node [anchor=north west][inner sep=0.75pt]  [font=\Large,rotate=-359.41] [align=left] {1};
\draw (454.88,10.07) node [anchor=north west][inner sep=0.75pt]  [font=\Large,rotate=-359.41] [align=left] {D};
\draw (204.88,15.07) node [anchor=north west][inner sep=0.75pt]  [font=\Large,rotate=-359.41] [align=left] {B};
\draw (111.88,14.07) node [anchor=north west][inner sep=0.75pt]  [font=\Large,rotate=-359.41] [align=left] {3};
\draw (110,135) node [anchor=north west][inner sep=0.75pt]  [color={rgb, 255:red, 74; green, 144; blue, 226 }  ,opacity=1 ]  {$T_{1}$};
\draw (162,89) node [anchor=north west][inner sep=0.75pt]  [color={rgb, 255:red, 74; green, 144; blue, 226 }  ,opacity=1 ]  {$T_{2}$};
\draw (270,121) node [anchor=north west][inner sep=0.75pt]  [color={rgb, 255:red, 74; green, 144; blue, 226 }  ,opacity=1 ]  {$T_{3}$};
\draw (375,66) node [anchor=north west][inner sep=0.75pt]  [color={rgb, 255:red, 74; green, 144; blue, 226 }  ,opacity=1 ]  {$ \begin{array}{l}
T_{4}\\
\end{array}$};
\end{tikzpicture}
\caption{A spanning forest for the tiered tree of Figure~\ref{fig1}}\label{fig5}
\end{figure}
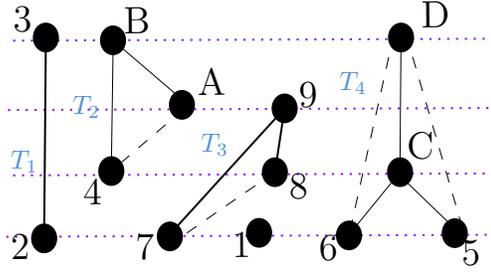

\begin{lemma}
    The elements $Z_{G \setminus T_{\textbf{act}}}$ where $G\setminus T_{\textbf{act}}$ is the graph obtained by removing the subgraph $T_{\textbf{act}}$ from $G$, spans the space $\mathcal{S}_{G}.$

\end{lemma} 

\begin{proof}
We prove the lemma by contradiction. Let $Z_{G \setminus T_{\textbf{act}}}$ do not span $\mathcal{S}_{G}$. Then there exists a maximal slim subgraph $H_{k}$ such that $Z_{H_{k}}$ can not be expressed as a linear combinations of the elements of $Z_{G\setminus T_{\textbf{act}}}$. Let $\mathcal{F}$ be the spannig forest of $G$. Let $T_{1}, \dots, T_{l}$ be the disjoint components of $\mathcal{F}$ (see Figure ~\ref{fig5}).  
 For each pair of vertices $(j,k) \in T_{j}$ with $j< k$ and $\textbf{t}(j) < \textbf{t}(k)$, we connect those vertices by an edge in $T_{j}$ for $j= 1,2, \dots, l$. Therefore, there is a sequence of subgraphs $H_{1}, \dots, H_{l}$ with ${T_{j}}$ as spannig trees for $j=1,2, \dots, l$. For each components $H_{j}$ there is an edge between a vertex of $H_{j}$ and the lowest labelled vertex $r$. Therefore after deleting any component $H_{j}$, there exist a path between every pair of vertices in the resulting graph. Hence $G \setminus H_{j}$ is connected. Thus each subgraphs $H_{j}$ are slim subgraphs of $G$. Now the edges of $G$ are ordered lexicographically, we  choose a maximal slim subgraph $H_{k}$. Then there exists a unique cycle $C=(e_{1}, \dots e_{r}, \dots, e_{l})$ in $G$ such that $H_{k} \cap C= \{e_{r}\}$. Since $r$ is minimal as well as the root of the graph $G$, we have $z_{e_{r}}$ is a linear combinations of $z_{e_{1}}, \dots, z_{e_{l}}$. For each slim subgraph $H_{j}$ there exists a path in the graph $G$ such that which contains vertices from each $H_{j}$s and the vertex $r$. Thus $Z_{H_{k}}$ is a linear combination of $Z_{H_{1}}, \dots  Z_{H_{l}}$ which leads to a contradiction. 
\end{proof}
\begin{Example} 
In Figure \ref{fig5}, the forest is consist of disjoint union of tiered trees $T_{1},\dots T_{4}$ (black edges) and the root $1$. The corresponding slim subgraphs are $H_{1}, \dots, H_{4}$ (after joining the compatible edges, dotted lines). After deleting any component $H_{i}$ from the complete tiered graph $G$, there exists a path which contains vertices of $G \setminus H_{i}$ and the vertex $1$ (as a common vertex). For instance, $P=24BA7981$ is a path after removing the component $H_{4}$. Therefore $G \setminus H_{i}$ is connected. Since the edges are lexicographically ordered by the end points, we consider $H_{4}$ as maximal slim subgraph of $G$. The vertex $1$ is adjacent to a vertex of the components $H_{1}, \dots H_{4}$. Therefore, $z_{1}= z_{D}+z_{A}+z_{B}+ \dots + z_{3}$.
\end{Example}
We have the following result regarding dimension of the algebra $\mathcal{C}_{G}$ where $G$ is a complete tiered graph in terms of the space $\mathcal{S}_{G}$
\begin{lemma}
    Let $G$ be a complete tiered graph then $\text{dim} 
     (\mathcal{C}^{k}_{G}) = \text{dim}(\mathcal{S}^{k}_{G})$.
\end{lemma}
\begin{proof}
 It is clear that the monomials $\phi_{H}= \prod_{e \in H} \phi_{e}$, where $H$ ranges over all slim subgraphs of a complete tiered graph $G$ form a linear basis of $\Phi_{G}$. Let us define a ring homomorphism $\beta: \mathbb{K}[x_{1}, \dots, x_{n}] \mapsto \mathbb{K}[z_{1}, \dots, z_{n}]$ by $\beta(x_{i})=z_{i}$. We need to show if $(b_{1}, \dots, b_{n})$ is a basis of $\mathcal{C}_{G}$ then $(\beta(b_{1}), \dots, \beta(b_{n}))$ is a basis of $\mathcal{S}_{G}$. Let $B_{H}$, where $H$ is slim subgraph of $G$, is the matrix form of $(b_{1}, \dots, b_{n})$. Since $H$ is a tiered graph and the map $\beta$ induce a graph homomorphism $\beta^{\prime}: H \mapsto H$. Therefore, $B_{H}$ is a matrix form of 
 $(\beta(b_{1}), \dots, \beta(b_{n}))$. Hence $dim(\mathcal{C}^{k}_{G})$ is equal to the rank of the matrix $B_{H}$, which is also equal to $dim( \mathcal{S}^{k}_{G})$. 
\end{proof}

In \cite{DGGE19}, the authors defined the statistic weight $\textbf{wt}(T) $ on labelled trees and showed that weight of a tiered tree is equal to the external activity of its complete tired graph. It would be a interesting problem to define an algebra on tiered graph (or simply connected graphs) which enumerates such weights of tiered trees (or labelled trees).

 \subsection{Dual graphs} 
Let $G$ be a connected tiered graph with tiereing function $\textbf{t}$. Dong and Yan \cite{DY} defined the dual graph $G^{*}$ of tiered graph $G$ with two tiers. In this section we extended the definition of dual graphs for higher tier. 
\begin{Def}
Let $G$ be a connected tiered graph with $m$ tiers, where $m \geqslant 2$. The dual graph $G^{*}$ of $G$ is the graph with $V(G^{*})= V(G)$ and $\textbf{t}(i)=(m+1)- \textbf{t}(n+1-i)$. Moreover, for $(i,j) \in E(G^{*})$ if and only if $(n+1-i, n+1-j) \in E(G)$.  
\end{Def}
\begin{lemma}
   Let $G$ be a connected tiered graph and $G^{*}$ be the dual graph of $G$. Then $G \cong G^{*}$. 
\end{lemma}
\begin{proof}
    From the definition of the dual graph we have $V(G^{*})=V(G)$,  where vertices are in total order. Define a map $\gamma : V(G)\mapsto V(G^{*}) $ by $\gamma(i)= n+1-i$. Since the adjacency of $G$ is preserved in $G^{*}$ under $\gamma$, the result follows.    
\end{proof}
Note that, the map $\gamma$ sends the minimal element of the graph $G$ to maximal element of $G^{*}$. Thus, the vertices that contribute to the externally active edges, their images may not contribute to the externally active edges of $G^{*}$. For example the tiered tree $T$ and its dual $T^{*}$ in Figure ~\ref{fig6}. The graph $G^{*}$ is obtained by flipping the graph $G$ from bottom to top.    

\begin{lemma}
   For a tiered graph $G$, let $F_{m}$ be the set of spannig tiered forest and  $F_{m^{*}}$ be the set of its dual. Then there is a bijection from $F_{m} \mapsto F_{m^{*}}.$
    
\end{lemma}
\begin{proof}
    Define the map $\theta: F_{m} \mapsto F_{m^{*}}$ by $\theta(F)=F^{*}$. The map $\theta$ is surjective since $(m^{*})^{*} =m$ and from the definition of dual tiered graph we have, $\theta(F^{*})= (F^{*})^{*}= F$. This map $\theta$ is also injective. Let $F^{*}_{1}=F^{*}_{2}$ then $\theta(F_{1})=F^{*}_{1}=F^{*}_{2}=\theta(F_{2})$. Also for  $\theta(F^{*}_{1})=\theta(F^{*}_{2})$, we have $F_{1}=\theta(F^{*}_{1})=\theta(F^{*}_{2})=F_{2}$. Hence the proof.
\end{proof}

\begin{proposition}
    For a complete tiered graph $G$ and let $F_{m}$ be the set of tiered forests and $F^{*}_{m^{*}}$ be the set of its dual.Then the algebras $\mathcal{B}^{F_{m}}_{G}$ and $\mathcal{B}^{F_{m^{*}}}_{G^{*}}$ are isomorphic. 
\end{proposition}
\begin{proof}
   The proof is straightforward.  From the above lemma we have bijection between tiered forests and its dual forests. Since we know the dual graph can be obtained by flipping the tiered graph and this bijection preserves the adjacency relation. 
\end{proof}
\subsection{Whitney's operation}
Now we define tiered graphs in terms of graph operations: \textit{identification, cleaving and twisting}. These operations are appeared as Whitney's 2-isomorphism theorem \cite{Ox}. 
\begin{itemize}
    \item [A.](\textbf{Identification}) Let $v_{1}$ and $v_{2}$ are two distinct vertices from different components. Then we can identify  $v_{1}$ and $v_{2}$ by a new vertex $v_{12}$. So, all those vertices which were adjacent to $v_{1}$ and $v_{2}$ is now adjacent to $v_{12}$.
    \item[B.](\textbf{Cleaving}) A graph only be cleft at cut-vertex or at a vertex incident with a loop. This is a reverse operation of identification.
    \item[C.](\textbf{Twisting)} Assume that $G$ be a graph obtained from two disjoint graphs $G_{1}$ and $G_{2}$ by identifying $u_{1}$ of $V(G_{1})$ and $u_{2}$ of $V(G_{2})$ as a new vertex $u$ of $G$ and identify $v_{1}$ of $G_{1}$ and $v_{2}$ of $G_{2}$ as a vertex $v$ of $G$. In a twisting about $\{u,v\}$, we identify $u_{1}$ with $v_{2}$ and $u_{2}$ with $v_{1}$. Thus, we obtain a new graph $G_{\circ}$ and call $G_{1}$ and $G_{2}$ be the pieces of the twisting. 
    
\end{itemize}
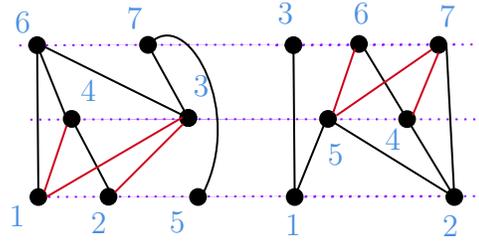
\begin{figure}
    \centering
    \tikzset{every picture/.style={line width=0.75pt}} 

\begin{tikzpicture}[x=0.35pt,y=0.35pt,yscale=-1,xscale=1]

\draw [color={rgb, 255:red, 144; green, 19; blue, 254 }  ,draw opacity=1 ][fill={rgb, 255:red, 144; green, 19; blue, 254 }  ,fill opacity=1 ][line width=0.75]  [dash pattern={on 0.84pt off 2.51pt}]  (101.79,44.66) -- (249.43,44.06) -- (573,43.6) ;
\draw [color={rgb, 255:red, 144; green, 19; blue, 254 }  ,draw opacity=1 ][fill={rgb, 255:red, 144; green, 19; blue, 254 }  ,fill opacity=1 ][line width=0.75]  [dash pattern={on 0.84pt off 2.51pt}]  (113.79,124.97) -- (261.43,124.36) -- (611,124.6) ;
\draw [color={rgb, 255:red, 144; green, 19; blue, 254 }  ,draw opacity=1 ][fill={rgb, 255:red, 144; green, 19; blue, 254 }  ,fill opacity=1 ][line width=0.75]  [dash pattern={on 0.84pt off 2.51pt}]  (114,208.62) -- (261.64,208.02) -- (583,207.6) ;
\draw  [fill={rgb, 255:red, 0; green, 0; blue, 0 }  ,fill opacity=1 ][line width=0.75]  (112.64,44.36) .. controls (112.64,39.65) and (116.44,35.83) .. (121.13,35.83) .. controls (125.81,35.83) and (129.61,39.65) .. (129.61,44.36) .. controls (129.61,49.07) and (125.81,52.89) .. (121.13,52.89) .. controls (116.44,52.89) and (112.64,49.07) .. (112.64,44.36) -- cycle ;
\draw  [fill={rgb, 255:red, 0; green, 0; blue, 0 }  ,fill opacity=1 ][line width=0.75]  (286.51,208.6) .. controls (286.51,203.89) and (290.31,200.07) .. (295,200.07) .. controls (299.69,200.07) and (303.49,203.89) .. (303.49,208.6) .. controls (303.49,213.31) and (299.69,217.13) .. (295,217.13) .. controls (290.31,217.13) and (286.51,213.31) .. (286.51,208.6) -- cycle ;
\draw  [fill={rgb, 255:red, 0; green, 0; blue, 0 }  ,fill opacity=1 ][line width=0.75]  (276,123.13) .. controls (276,118.42) and (279.8,114.6) .. (284.49,114.6) .. controls (289.17,114.6) and (292.97,118.42) .. (292.97,123.13) .. controls (292.97,127.84) and (289.17,131.66) .. (284.49,131.66) .. controls (279.8,131.66) and (276,127.84) .. (276,123.13) -- cycle ;
\draw  [fill={rgb, 255:red, 0; green, 0; blue, 0 }  ,fill opacity=1 ][line width=0.75]  (150,124.13) .. controls (150,119.42) and (153.8,115.6) .. (158.49,115.6) .. controls (163.17,115.6) and (166.97,119.42) .. (166.97,124.13) .. controls (166.97,128.84) and (163.17,132.66) .. (158.49,132.66) .. controls (153.8,132.66) and (150,128.84) .. (150,124.13) -- cycle ;
\draw  [fill={rgb, 255:red, 0; green, 0; blue, 0 }  ,fill opacity=1 ][line width=0.75]  (189.82,208.32) .. controls (189.82,203.61) and (193.62,199.79) .. (198.31,199.79) .. controls (202.99,199.79) and (206.79,203.61) .. (206.79,208.32) .. controls (206.79,213.03) and (202.99,216.85) .. (198.31,216.85) .. controls (193.62,216.85) and (189.82,213.03) .. (189.82,208.32) -- cycle ;
\draw  [fill={rgb, 255:red, 0; green, 0; blue, 0 }  ,fill opacity=1 ][line width=0.75]  (232.46,44.06) .. controls (232.46,39.35) and (236.26,35.53) .. (240.95,35.53) .. controls (245.63,35.53) and (249.43,39.35) .. (249.43,44.06) .. controls (249.43,48.77) and (245.63,52.59) .. (240.95,52.59) .. controls (236.26,52.59) and (232.46,48.77) .. (232.46,44.06) -- cycle ;
\draw  [fill={rgb, 255:red, 0; green, 0; blue, 0 }  ,fill opacity=1 ][line width=0.75]  (114,208.62) .. controls (114,203.91) and (117.8,200.09) .. (122.49,200.09) .. controls (127.17,200.09) and (130.97,203.91) .. (130.97,208.62) .. controls (130.97,213.33) and (127.17,217.15) .. (122.49,217.15) .. controls (117.8,217.15) and (114,213.33) .. (114,208.62) -- cycle ;
\draw [color={rgb, 255:red, 208; green, 2; blue, 27 }  ,draw opacity=1 ]   (153,131.6) -- (129,202.6) ;
\draw    (158.49,124.13) -- (198.31,199.79) ;
\draw  [fill={rgb, 255:red, 0; green, 0; blue, 0 }  ,fill opacity=1 ][line width=0.75]  (276,123.13) .. controls (276,118.42) and (279.8,114.6) .. (284.49,114.6) .. controls (289.17,114.6) and (292.97,118.42) .. (292.97,123.13) .. controls (292.97,127.84) and (289.17,131.66) .. (284.49,131.66) .. controls (279.8,131.66) and (276,127.84) .. (276,123.13) -- cycle ;
\draw    (121.13,44.36) -- (122.49,208.62) ;
\draw    (121.13,44.36) -- (155,128.6) ;
\draw    (126,44.6) -- (284.49,123.13) ;
\draw    (240.95,44.06) -- (284.49,123.13) ;
\draw    (246.95,39.06) .. controls (286.95,9.06) and (344,116.6) .. (301,203.6) ;
\draw [color={rgb, 255:red, 208; green, 2; blue, 27 }  ,draw opacity=1 ]   (130.97,208.62) -- (278,122.6) ;
\draw [color={rgb, 255:red, 208; green, 2; blue, 27 }  ,draw opacity=1 ]   (205,203.6) -- (279,129.6) ;
\draw [color={rgb, 255:red, 144; green, 19; blue, 254 }  ,draw opacity=1 ][fill={rgb, 255:red, 144; green, 19; blue, 254 }  ,fill opacity=1 ][line width=0.75]  [dash pattern={on 0.84pt off 2.51pt}]  (378.79,44.66) -- (526.43,44.06) -- (600,43.6) ;
\draw [color={rgb, 255:red, 144; green, 19; blue, 254 }  ,draw opacity=1 ][fill={rgb, 255:red, 144; green, 19; blue, 254 }  ,fill opacity=1 ][line width=0.75]  [dash pattern={on 0.84pt off 2.51pt}]  (391,208.62) -- (538.64,208.02) -- (610,207.6) ;
\draw  [fill={rgb, 255:red, 0; green, 0; blue, 0 }  ,fill opacity=1 ][line width=0.75]  (389.64,44.36) .. controls (389.64,39.65) and (393.44,35.83) .. (398.13,35.83) .. controls (402.81,35.83) and (406.61,39.65) .. (406.61,44.36) .. controls (406.61,49.07) and (402.81,52.89) .. (398.13,52.89) .. controls (393.44,52.89) and (389.64,49.07) .. (389.64,44.36) -- cycle ;
\draw  [fill={rgb, 255:red, 0; green, 0; blue, 0 }  ,fill opacity=1 ][line width=0.75]  (563.51,208.6) .. controls (563.51,203.89) and (567.31,200.07) .. (572,200.07) .. controls (576.69,200.07) and (580.49,203.89) .. (580.49,208.6) .. controls (580.49,213.31) and (576.69,217.13) .. (572,217.13) .. controls (567.31,217.13) and (563.51,213.31) .. (563.51,208.6) -- cycle ;
\draw  [fill={rgb, 255:red, 0; green, 0; blue, 0 }  ,fill opacity=1 ][line width=0.75]  (427,124.13) .. controls (427,119.42) and (430.8,115.6) .. (435.49,115.6) .. controls (440.17,115.6) and (443.97,119.42) .. (443.97,124.13) .. controls (443.97,128.84) and (440.17,132.66) .. (435.49,132.66) .. controls (430.8,132.66) and (427,128.84) .. (427,124.13) -- cycle ;
\draw  [fill={rgb, 255:red, 0; green, 0; blue, 0 }  ,fill opacity=1 ][line width=0.75]  (546.46,44.06) .. controls (546.46,39.35) and (550.26,35.53) .. (554.95,35.53) .. controls (559.63,35.53) and (563.43,39.35) .. (563.43,44.06) .. controls (563.43,48.77) and (559.63,52.59) .. (554.95,52.59) .. controls (550.26,52.59) and (546.46,48.77) .. (546.46,44.06) -- cycle ;
\draw  [fill={rgb, 255:red, 0; green, 0; blue, 0 }  ,fill opacity=1 ][line width=0.75]  (391,208.62) .. controls (391,203.91) and (394.8,200.09) .. (399.49,200.09) .. controls (404.17,200.09) and (407.97,203.91) .. (407.97,208.62) .. controls (407.97,213.33) and (404.17,217.15) .. (399.49,217.15) .. controls (394.8,217.15) and (391,213.33) .. (391,208.62) -- cycle ;
\draw [color={rgb, 255:red, 0; green, 0; blue, 0 }  ,draw opacity=1 ]   (432,128.6) -- (399.49,208.62) ;
\draw    (435.49,124.13) -- (572,208.6) ;
\draw  [fill={rgb, 255:red, 0; green, 0; blue, 0 }  ,fill opacity=1 ][line width=0.75]  (512.46,124.36) .. controls (512.46,119.65) and (516.26,115.83) .. (520.95,115.83) .. controls (525.63,115.83) and (529.43,119.65) .. (529.43,124.36) .. controls (529.43,129.07) and (525.63,132.89) .. (520.95,132.89) .. controls (516.26,132.89) and (512.46,129.07) .. (512.46,124.36) -- cycle ;
\draw    (398.13,44.36) -- (399.49,208.62) ;
\draw  [fill={rgb, 255:red, 0; green, 0; blue, 0 }  ,fill opacity=1 ][line width=0.75]  (460.46,43.06) .. controls (460.46,38.35) and (464.26,34.53) .. (468.95,34.53) .. controls (473.63,34.53) and (477.43,38.35) .. (477.43,43.06) .. controls (477.43,47.77) and (473.63,51.59) .. (468.95,51.59) .. controls (464.26,51.59) and (460.46,47.77) .. (460.46,43.06) -- cycle ;
\draw    (520.95,124.36) -- (572,208.6) ;
\draw    (469.9,40.13) -- (520.95,124.36) ;
\draw    (563.43,44.06) -- (572,208.6) ;
\draw [color={rgb, 255:red, 208; green, 2; blue, 27 }  ,draw opacity=1 ]   (464.51,48.58) -- (441,117.6) ;
\draw [color={rgb, 255:red, 208; green, 2; blue, 27 }  ,draw opacity=1 ]   (547,49.6) -- (444.43,120.61) ;
\draw [color={rgb, 255:red, 208; green, 2; blue, 27 }  ,draw opacity=1 ]   (554.95,52.59) -- (527,120.13) ;

\draw (94,5) node [anchor=north west][inner sep=0.75pt]  [font=\large,color={rgb, 255:red, 74; green, 144; blue, 226 }  ,opacity=1 ]  {$6$};
\draw (192,220) node [anchor=north west][inner sep=0.75pt]  [color={rgb, 255:red, 74; green, 144; blue, 226 }  ,opacity=1 ]  {$$};
\draw (469,220) node [anchor=north west][inner sep=0.75pt]  [color={rgb, 255:red, 74; green, 144; blue, 226 }  ,opacity=1 ]  {$$};
\draw (386,225) node [anchor=north west][inner sep=0.75pt]  [font=\large,color={rgb, 255:red, 74; green, 144; blue, 226 }  ,opacity=1 ]  {$1$};
\draw (494,132) node [anchor=north west][inner sep=0.75pt]  [font=\large,color={rgb, 255:red, 74; green, 144; blue, 226 }  ,opacity=1 ]  {$4$};
\draw (432,148) node [anchor=north west][inner sep=0.75pt]  [font=\large,color={rgb, 255:red, 74; green, 144; blue, 226 }  ,opacity=1 ]  {$5$};
\draw (553,-1) node [anchor=north west][inner sep=0.75pt]  [font=\large,color={rgb, 255:red, 74; green, 144; blue, 226 }  ,opacity=1 ]  {$7$};
\draw (460,-4) node [anchor=north west][inner sep=0.75pt]  [font=\large,color={rgb, 255:red, 74; green, 144; blue, 226 }  ,opacity=1 ]  {$6$};
\draw (378,-4) node [anchor=north west][inner sep=0.75pt]  [font=\large,color={rgb, 255:red, 74; green, 144; blue, 226 }  ,opacity=1 ]  {$3$};
\draw (261,221) node [anchor=north west][inner sep=0.75pt]  [font=\large,color={rgb, 255:red, 74; green, 144; blue, 226 }  ,opacity=1 ]  {$5$};
\draw (176,223) node [anchor=north west][inner sep=0.75pt]  [font=\large,color={rgb, 255:red, 74; green, 144; blue, 226 }  ,opacity=1 ]  {$2$};
\draw (88,215) node [anchor=north west][inner sep=0.75pt]  [font=\large,color={rgb, 255:red, 74; green, 144; blue, 226 }  ,opacity=1 ]  {$1$};
\draw (287,74) node [anchor=north west][inner sep=0.75pt]  [font=\large,color={rgb, 255:red, 74; green, 144; blue, 226 }  ,opacity=1 ]  {$3$};
\draw (165,79) node [anchor=north west][inner sep=0.75pt]  [font=\large,color={rgb, 255:red, 74; green, 144; blue, 226 }  ,opacity=1 ]  {$4$};
\draw (215,1) node [anchor=north west][inner sep=0.75pt]  [font=\large,color={rgb, 255:red, 74; green, 144; blue, 226 }  ,opacity=1 ]  {$7$};
\draw (556,226) node [anchor=north west][inner sep=0.75pt]  [font=\large,color={rgb, 255:red, 74; green, 144; blue, 226 }  ,opacity=1 ]  {$2$};
\end{tikzpicture}
\caption{A tiered graph (left) with red edges are externally active and its dual graph (right) with red edges as the image.}\label{fig6}
    \label{fig:enter-label}
\end{figure}

If two graphs $G$ and $H$ obtained from each other from the above operations we say $G$ is 2-isomorphic to $H$ and also their graphic matroids $M(G)$ and $M(H)$ are isomorphic. We generalized the notion of identification for tiered graphs. 
\begin{Def}
    Let $G_{1}=(V_{1}, E_{1},\textbf{t}_{1} )$ and $G_{2}=(V_{2}, E_{2}, \textbf{t}_{2})$ be two tiered graphs. We denote the identification operation by $\bullet$ and write $i=i^{\prime} \bullet i^{\dprime}$ for $i ^{\prime} \in V_{1}$ and $i^{\dprime} \in V_{2}$. Therefore, if $i^{\prime} \in [m^{\prime}]$ and $i^{\dprime} \in [m^{\dprime}]$, we define the identification $i$ satisfying $i^{\prime} \prec i $ and $i^{\dprime} \prec i $.  Similarly we define the identification operation to the tiering function. For two disjoint sets $A=\{1^{\prime}, \dots,m^{\prime}\}$ and $B=\{1^{\dprime}, \dots, m^{\dprime}\}$ we write identification $A \bullet B= \{1^{\prime} \bullet 1^{\dprime}, \dots, m^{\prime} \bullet m^{\dprime}\}= [m^{\prime} \bullet m^{\dprime}]$. Thus identification of tiers is a function $\textbf{t}: V_{1} \bullet V_{2} \mapsto [m^{\prime} \bullet m^{\dprime}]
      $ such that for $j^{\prime} \in V_{1}$ and $j^{\prime}$ is adjacent to $i$ if and only if $j^{\prime} \prec i$ and $\textbf{t}(j^{\prime}) \prec \textbf{t}(i) $. Similarly for the vertices of $G_{2}$. We write identification of tiers by $\textbf{t}(i)= \textbf{t}_{1}(i^{\prime}) \bullet \textbf{t}_{2}(i^{\dprime})$, which we visualized as concatenation of tiers. For an illustration see Figure ~\ref{fig7} and ~\ref{fig8}. 
    
\end{Def}
We remark that such operations on tiering functions do not harm the basic properties of tiered graph. The resulting graph is also a tiered graph. 
\begin{lemma}\label{TT}
    Let $G_{1}=(V_{1}, E_{1},\textbf{t}_{1} )$ and $G_{2}=(V_{2}, E_{2}, \textbf{t}_{2})$ be two tiered graphs. Then $G_{1}$ and $G_{2}$ are closed under Whitney's operations. 
\end{lemma}
\begin{proof}
    Let $G_{1}$ and $G_{2}$ are two disjoint tiered graphs. Let a vertex $i^{\prime}$ of $G_{1}$ lies in tier $\textbf{t}(i^{\prime})$ and $i^{\dprime}$ of $G_{2}$ in tier $\textbf{t}(i^{\dprime})$. Identifying $i^{\prime}$ and $i^{\dprime}$ by new labelling $i$. That is, we write$i= i^{\prime} \bullet i^{\dprime}$ and for tiering function write $\textbf{t}(i)= \textbf{t}_{1}(i^{\prime}) \bullet \textbf{t}_{2}(i^{\dprime}) $. By the definition, the resulting graph is a tiered graph $G$. Also we assume that one of $G_{1}$ or $G_{2}$ might have lower tier. Without loss of generality assume $G_{1}$ has lower tier \textit{i.e} $m^{\prime} \leqslant m^{\dprime}$. Then we will add one more tier to $G_{1}$ which will call \textit{empty tier} (tier contains no vertex). Thus from the construction and above definition, the first two conditions of Whitney's isomorphism is satisfied. For example, see Figures~\ref{fig7} and ~\ref{fig8} below.
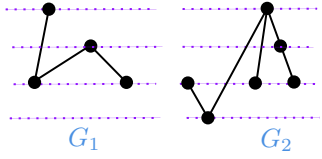
\begin{figure}
    \centering
\tikzset{every picture/.style={line width=0.75pt}} 

\begin{tikzpicture}[x=0.25pt,y=0.25pt,yscale=-1,xscale=1]

\draw [color={rgb, 255:red, 144; green, 19; blue, 254 }  ,draw opacity=1 ][fill={rgb, 255:red, 144; green, 19; blue, 254 }  ,fill opacity=1 ][line width=0.75]  [dash pattern={on 0.84pt off 2.51pt}]  (101.79,44.66) -- (329.43,44.06) (369.43,44.06)-- (573,43.6) ;
\draw [color={rgb, 255:red, 144; green, 19; blue, 254 }  ,draw opacity=1 ][fill={rgb, 255:red, 144; green, 19; blue, 254 }  ,fill opacity=1 ][line width=0.75]  [dash pattern={on 0.84pt off 2.51pt}]  (113.79,157.97) -- (329.43,157.36) (365.43,157.36)-- (570,157.6) ;
\draw [color={rgb, 255:red, 144; green, 19; blue, 254 }  ,draw opacity=1 ][fill={rgb, 255:red, 144; green, 19; blue, 254 }  ,fill opacity=1 ][line width=0.75]  [dash pattern={on 0.84pt off 2.51pt}]  (108,209.62) -- (335.64,209.02) (375.64,209.02)-- (577,208.6) ;
\draw  [fill={rgb, 255:red, 0; green, 0; blue, 0 }  ,fill opacity=1 ][line width=0.75]  (489,43.13) .. controls (489,38.42) and (492.8,34.6) .. (497.49,34.6) .. controls (502.17,34.6) and (505.97,38.42) .. (505.97,43.13) .. controls (505.97,47.84) and (502.17,51.66) .. (497.49,51.66) .. controls (492.8,51.66) and (489,47.84) .. (489,43.13) -- cycle ;
\draw  [fill={rgb, 255:red, 0; green, 0; blue, 0 }  ,fill opacity=1 ][line width=0.75]  (158.64,44.36) .. controls (158.64,39.65) and (162.44,35.83) .. (167.13,35.83) .. controls (171.81,35.83) and (175.61,39.65) .. (175.61,44.36) .. controls (175.61,49.07) and (171.81,52.89) .. (167.13,52.89) .. controls (162.44,52.89) and (158.64,49.07) .. (158.64,44.36) -- cycle ;
\draw  [fill={rgb, 255:red, 0; green, 0; blue, 0 }  ,fill opacity=1 ][line width=0.75]  (529.51,154.6) .. controls (529.51,149.89) and (533.31,146.07) .. (538,146.07) .. controls (542.69,146.07) and (546.49,149.89) .. (546.49,154.6) .. controls (546.49,159.31) and (542.69,163.13) .. (538,163.13) .. controls (533.31,163.13) and (529.51,159.31) .. (529.51,154.6) -- cycle ;
\draw  [fill={rgb, 255:red, 0; green, 0; blue, 0 }  ,fill opacity=1 ][line width=0.75]  (276,155.13) .. controls (276,150.42) and (279.8,146.6) .. (284.49,146.6) .. controls (289.17,146.6) and (292.97,150.42) .. (292.97,155.13) .. controls (292.97,159.84) and (289.17,163.66) .. (284.49,163.66) .. controls (279.8,163.66) and (276,159.84) .. (276,155.13) -- cycle ;
\draw  [fill={rgb, 255:red, 0; green, 0; blue, 0 }  ,fill opacity=1 ][line width=0.75]  (137,155.13) .. controls (137,150.42) and (140.8,146.6) .. (145.49,146.6) .. controls (150.17,146.6) and (153.97,150.42) .. (153.97,155.13) .. controls (153.97,159.84) and (150.17,163.66) .. (145.49,163.66) .. controls (140.8,163.66) and (137,159.84) .. (137,155.13) -- cycle ;
\draw  [fill={rgb, 255:red, 0; green, 0; blue, 0 }  ,fill opacity=1 ][line width=0.75]  (471,156.13) .. controls (471,151.42) and (474.8,147.6) .. (479.49,147.6) .. controls (484.17,147.6) and (487.97,151.42) .. (487.97,156.13) .. controls (487.97,160.84) and (484.17,164.66) .. (479.49,164.66) .. controls (474.8,164.66) and (471,160.84) .. (471,156.13) -- cycle ;
\draw  [fill={rgb, 255:red, 0; green, 0; blue, 0 }  ,fill opacity=1 ][line width=0.75]  (222,100.13) .. controls (222,95.42) and (225.8,91.6) .. (230.49,91.6) .. controls (235.17,91.6) and (238.97,95.42) .. (238.97,100.13) .. controls (238.97,104.84) and (235.17,108.66) .. (230.49,108.66) .. controls (225.8,108.66) and (222,104.84) .. (222,100.13) -- cycle ;
\draw  [fill={rgb, 255:red, 0; green, 0; blue, 0 }  ,fill opacity=1 ][line width=0.75]  (399.35,208.81) .. controls (399.35,204.1) and (403.15,200.28) .. (407.84,200.28) .. controls (412.52,200.28) and (416.32,204.1) .. (416.32,208.81) .. controls (416.32,213.52) and (412.52,217.34) .. (407.84,217.34) .. controls (403.15,217.34) and (399.35,213.52) .. (399.35,208.81) -- cycle ;
\draw  [fill={rgb, 255:red, 0; green, 0; blue, 0 }  ,fill opacity=1 ][line width=0.75]  (509,100.13) .. controls (509,95.42) and (512.8,91.6) .. (517.49,91.6) .. controls (522.17,91.6) and (525.97,95.42) .. (525.97,100.13) .. controls (525.97,104.84) and (522.17,108.66) .. (517.49,108.66) .. controls (512.8,108.66) and (509,104.84) .. (509,100.13) -- cycle ;
\draw    (167.13,44.36) -- (145.49,155.13) ;
\draw    (228.49,100.13) -- (148,151.6) ;
\draw    (230.49,100.13) -- (284.49,155.13) ;
\draw    (497.49,43.13) -- (481.26,153.86) ;
\draw    (497.49,43.13) -- (538,154.6) ;
\draw [color={rgb, 255:red, 144; green, 19; blue, 254 }  ,draw opacity=1 ][fill={rgb, 255:red, 144; green, 19; blue, 254 }  ,fill opacity=1 ][line width=0.75]  [dash pattern={on 0.84pt off 2.51pt}]  (111.79,101.97) -- (339.43,101.36) (369.43,101.36)-- (568,101.6) ;
\draw    (497.49,43.13) -- (407.84,208.81) ;
\draw  [fill={rgb, 255:red, 0; green, 0; blue, 0 }  ,fill opacity=1 ][line width=0.75]  (369,156.13) .. controls (369,151.42) and (372.8,147.6) .. (377.49,147.6) .. controls (382.17,147.6) and (385.97,151.42) .. (385.97,156.13) .. controls (385.97,160.84) and (382.17,164.66) .. (377.49,164.66) .. controls (372.8,164.66) and (369,160.84) .. (369,156.13) -- cycle ;
\draw    (377.49,156.13) -- (407.84,208.81) ;

\draw (483,221) node [anchor=north west][inner sep=0.75pt]  [color={rgb, 255:red, 74; green, 144; blue, 226 }  ,opacity=1 ]  {$G_{2}$};
\draw (192,220) node [anchor=north west][inner sep=0.75pt]  [color={rgb, 255:red, 74; green, 144; blue, 226 }  ,opacity=1 ]  {$G_{1}$};
\end{tikzpicture}
\caption{Two tiered graphs $G_1$ and $G_2$.}\label{fig7}
\label{tg1}
\end{figure}
\begin{figure}
    \centering
   \tikzset{every picture/.style={line width=0.75pt}} 
\begin{tikzpicture}[x=0.35pt,y=0.35pt,yscale=-1,xscale=1]

\draw [color={rgb, 255:red, 144; green, 19; blue, 254 }  ,draw opacity=1 ][fill={rgb, 255:red, 144; green, 19; blue, 254 }  ,fill opacity=1 ][line width=0.75]  [dash pattern={on 0.84pt off 2.51pt}]  (101.79,44.66) -- (249.43,44.06) -- (573,43.6) ;
\draw [color={rgb, 255:red, 144; green, 19; blue, 254 }  ,draw opacity=1 ][fill={rgb, 255:red, 144; green, 19; blue, 254 }  ,fill opacity=1 ][line width=0.75]  [dash pattern={on 0.84pt off 2.51pt}]  (113.79,156.97) -- (261.43,156.36) -- (570,156.6) ;
\draw [color={rgb, 255:red, 144; green, 19; blue, 254 }  ,draw opacity=1 ][fill={rgb, 255:red, 144; green, 19; blue, 254 }  ,fill opacity=1 ][line width=0.75]  [dash pattern={on 0.84pt off 2.51pt}]  (114,208.62) -- (261.64,208.02) -- (583,207.6) ;
\draw  [fill={rgb, 255:red, 0; green, 0; blue, 0 }  ,fill opacity=1 ][line width=0.75]  (489,43.13) .. controls (489,38.42) and (492.8,34.6) .. (497.49,34.6) .. controls (502.17,34.6) and (505.97,38.42) .. (505.97,43.13) .. controls (505.97,47.84) and (502.17,51.66) .. (497.49,51.66) .. controls (492.8,51.66) and (489,47.84) .. (489,43.13) -- cycle ;
\draw  [fill={rgb, 255:red, 0; green, 0; blue, 0 }  ,fill opacity=1 ][line width=0.75]  (158.64,44.36) .. controls (158.64,39.65) and (162.44,35.83) .. (167.13,35.83) .. controls (171.81,35.83) and (175.61,39.65) .. (175.61,44.36) .. controls (175.61,49.07) and (171.81,52.89) .. (167.13,52.89) .. controls (162.44,52.89) and (158.64,49.07) .. (158.64,44.36) -- cycle ;
\draw  [fill={rgb, 255:red, 0; green, 0; blue, 0 }  ,fill opacity=1 ][line width=0.75]  (529.51,154.6) .. controls (529.51,149.89) and (533.31,146.07) .. (538,146.07) .. controls (542.69,146.07) and (546.49,149.89) .. (546.49,154.6) .. controls (546.49,159.31) and (542.69,163.13) .. (538,163.13) .. controls (533.31,163.13) and (529.51,159.31) .. (529.51,154.6) -- cycle ;
\draw  [fill={rgb, 255:red, 0; green, 0; blue, 0 }  ,fill opacity=1 ][line width=0.75]  (276,155.13) .. controls (276,150.42) and (279.8,146.6) .. (284.49,146.6) .. controls (289.17,146.6) and (292.97,150.42) .. (292.97,155.13) .. controls (292.97,159.84) and (289.17,163.66) .. (284.49,163.66) .. controls (279.8,163.66) and (276,159.84) .. (276,155.13) -- cycle ;
\draw  [fill={rgb, 255:red, 0; green, 0; blue, 0 }  ,fill opacity=1 ][line width=0.75]  (137,155.13) .. controls (137,150.42) and (140.8,146.6) .. (145.49,146.6) .. controls (150.17,146.6) and (153.97,150.42) .. (153.97,155.13) .. controls (153.97,159.84) and (150.17,163.66) .. (145.49,163.66) .. controls (140.8,163.66) and (137,159.84) .. (137,155.13) -- cycle ;
\draw  [fill={rgb, 255:red, 0; green, 0; blue, 0 }  ,fill opacity=1 ][line width=0.75]  (471,156.13) .. controls (471,151.42) and (474.8,147.6) .. (479.49,147.6) .. controls (484.17,147.6) and (487.97,151.42) .. (487.97,156.13) .. controls (487.97,160.84) and (484.17,164.66) .. (479.49,164.66) .. controls (474.8,164.66) and (471,160.84) .. (471,156.13) -- cycle ;
\draw  [fill={rgb, 255:red, 0; green, 0; blue, 0 }  ,fill opacity=1 ][line width=0.75]  (222,100.13) .. controls (222,95.42) and (225.8,91.6) .. (230.49,91.6) .. controls (235.17,91.6) and (238.97,95.42) .. (238.97,100.13) .. controls (238.97,104.84) and (235.17,108.66) .. (230.49,108.66) .. controls (225.8,108.66) and (222,104.84) .. (222,100.13) -- cycle ;
\draw  [fill={rgb, 255:red, 0; green, 0; blue, 0 }  ,fill opacity=1 ][line width=0.75]  (306.35,207.81) .. controls (306.35,203.1) and (310.15,199.28) .. (314.84,199.28) .. controls (319.52,199.28) and (323.32,203.1) .. (323.32,207.81) .. controls (323.32,212.52) and (319.52,216.34) .. (314.84,216.34) .. controls (310.15,216.34) and (306.35,212.52) .. (306.35,207.81) -- cycle ;
\draw  [fill={rgb, 255:red, 0; green, 0; blue, 0 }  ,fill opacity=1 ][line width=0.75]  (509,100.13) .. controls (509,95.42) and (512.8,91.6) .. (517.49,91.6) .. controls (522.17,91.6) and (525.97,95.42) .. (525.97,100.13) .. controls (525.97,104.84) and (522.17,108.66) .. (517.49,108.66) .. controls (512.8,108.66) and (509,104.84) .. (509,100.13) -- cycle ;
\draw    (167.13,44.36) -- (145.49,155.13) ;
\draw    (228.49,100.13) -- (148,151.6) ;
\draw    (230.49,100.13) -- (284.49,155.13) ;
\draw    (497.49,43.13) -- (481.26,153.86) ;
\draw    (497.49,43.13) -- (538,154.6) ;
\draw [color={rgb, 255:red, 144; green, 19; blue, 254 }  ,draw opacity=1 ][fill={rgb, 255:red, 144; green, 19; blue, 254 }  ,fill opacity=1 ][line width=0.75]  [dash pattern={on 0.84pt off 2.51pt}]  (111.79,101.97) -- (259.43,101.36) -- (568,101.6) ;
\draw    (497.49,43.13) -- (314.84,207.81) ;
\draw  [fill={rgb, 255:red, 0; green, 0; blue, 0 }  ,fill opacity=1 ][line width=0.75]  (276,155.13) .. controls (276,150.42) and (279.8,146.6) .. (284.49,146.6) .. controls (289.17,146.6) and (292.97,150.42) .. (292.97,155.13) .. controls (292.97,159.84) and (289.17,163.66) .. (284.49,163.66) .. controls (279.8,163.66) and (276,159.84) .. (276,155.13) -- cycle ;
\draw    (284.49,155.13) -- (314.84,207.81) ;
\draw [line width=1.5]  [dash pattern={on 1.69pt off 2.76pt}]  (284.49,155.13) -- (497.49,43.13) ;
\draw [line width=1.5]  [dash pattern={on 1.69pt off 2.76pt}]  (314.84,207.81) -- (145.49,155.13) ;
\draw [line width=1.5]  [dash pattern={on 1.69pt off 2.76pt}]  (477,156.6) -- (319.49,211.13) ;

\draw (307,237) node [anchor=north west][inner sep=0.75pt]  [font=\large,color={rgb, 255:red, 74; green, 144; blue, 226 }  ,opacity=1 ]  {$G$};
\draw (192,220) node [anchor=north west][inner sep=0.75pt]  [color={rgb, 255:red, 74; green, 144; blue, 226 }  ,opacity=1 ]  {};
\end{tikzpicture}
\caption{Tiered graph $G$ after the identification of $G_1$ and $G_2$ of Figure~\ref{tg1}.}\label{fig8}
\end{figure}
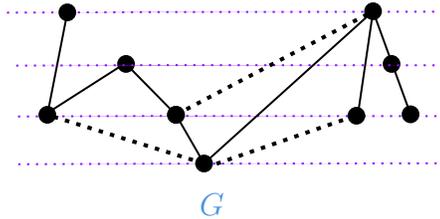
    Now we prove the twisting operation for tiered graphs. Consider two graphs $G_{1}$ and $G_{2}$. We identify two vertices from different tiers. Consider the vertices  $i^{\prime}$ and $j^{\prime}$ of $G_{1}$ and $i^{\dprime}$ and $j^{\dprime}$ of $G_{2}$. We  twist one of the graphs such that $i^{\prime}$ will be identified with $j^{\dprime}$ and $j^{\prime}$ with $i^{\dprime}$. The resulting graph $G$ is a tiered graph such that for $k \in V(G)$, either $k^{\prime} \prec (j^{\prime} \bullet i^{\dprime})$ and $\textbf{t}(k) \prec \textbf{t}((j^{\prime} \bullet i^{\dprime})$ or $(j^{\prime} \bullet i^{\dprime}) \prec k$ and $ \textbf{t}((j^{\prime} \bullet i^{\dprime}) \prec \textbf{t}(k)$. Thus the resulting graph is a tiered graph  using Whitney's twisting operation. 
    
\end{proof}
 
\begin{proposition}
    If two connected tired graphs $G$ and $H$ can be formed via Whitney's 2-isomorphism theorem then we have $\mathcal{C}^{F^{1}_{m}}_{G}$ $\cong$ $\mathcal{C}^{F^{2}_{m}}_{H}$.
\end{proposition}
\begin{proof}
    Combining the Lemma~\ref{TT} and Lemma 3 of \cite{N17}. We just add that when identify two vertices $i^{\prime}$ and $i^{\dprime}$, the variable $X_{i}$ can be written as $X_{i}= \pm X_{j^{\prime}} \pm X_{j^{\dprime}}$. Thus $X_{i}$ belongs to a linear space spanned by $< X_{1^{\prime}}, \dots, X_{(i^{\prime}-1)}, \dots, X_{m^{\prime}}, X_{1^{\dprime}}, \dots, X_{(j^{\dprime}+1)}, \dots, X_{m^{\dprime}}>   $. We have the desired result. 
\end{proof}
\bibliography{myref}

\end{document}